\documentclass{amsart}

\usepackage{amsmath}
\usepackage{amssymb}
\usepackage{amsthm}
\usepackage{mathscinet}
\usepackage{graphicx}
\usepackage{subcaption}
\usepackage{enumitem}
\usepackage{rotating}
\usepackage{tikz}
\usetikzlibrary{arrows}

\theoremstyle{plain}
\newtheorem{thm}{Theorem}[section]

\newtheorem{lem}[thm]{Lemma}
\newtheorem{cor}[thm]{Corollary}

\newtheorem{propalph}{Proposition}
\newtheorem{lemalph}[propalph]{Lemma}

\theoremstyle{definition}
\newtheorem{defn}[thm]{Definition}
\newtheorem*{defn*}{Definition}
\newtheorem{claim}{Claim}
\newtheorem{question}{Question}
\newtheorem*{question*}{Question}

\theoremstyle{definition}

\newcommand{\trunc}[2]{\left\lfloor #1 \right\rfloor_{#2}}
\newcommand{\mesh}{\mathrm{mesh}}
\newcommand{\diam}{\mathrm{diam}}
\newcommand{\intr}{\mathrm{Int}}

\begin{document}

\title[Nadler-Quinn problem]{The Nadler-Quinn problem\\on accessible points of arc-like continua}
\author{Andrea Ammerlaan \and Ana Anu\v{s}i\'{c} \and Logan C. Hoehn}
\date{\today}

\address{Nipissing University, Department of Computer Science \& Mathematics, 100 College Drive, Box 5002, North Bay, Ontario, Canada, P1B 8L7}
\email{ajammerlaan879@my.nipissingu.ca}
\email{anaa@nipissingu.ca}
\email{loganh@nipissingu.ca}

\thanks{This work was supported by NSERC grant RGPIN-2019-05998}


\begin{abstract}
We show that if $X$ is an arc-like continuum, then for every point $x \in X$ there is a plane embedding of $X$ in which $x$ is an accessible point.  This answers a question posed by Nadler in 1972, which has become known as the Nadler-Quinn problem in continuum theory.  Towards this end, we develop the theories of truncations and contour factorizations of interval maps.  As a corollary, we answer a question of Mayer from 1982 about inequivalent plane embeddings of indecomposable arc-like continua.
\end{abstract}

\maketitle

\section{Introduction}
\label{sec:intro}

Let $X$ be a \emph{continuum} (a compact, connected metric space) which is contained in the Euclidean plane $\mathbb{R}^2$.  A point $x \in X$ is \emph{accessible} if there exists a path $\gamma \colon [0,1] \to \mathbb{R}^2$ such that $\gamma(t) \notin X$ for all $0 \leq t < 1$ and $\gamma(1) = x$.  This concept was introduced by Schoenflies \cite{schoenflies1906} in 1906, in his classic work on Jordan curves in the plane, and has been a fundamental part of plane topology ever since.

Accessible points feature prominently within Carath\'{e}odory's theory of prime ends \cite{caratheodory1912} (see \cite{milnor2006} for a modern treatment).  In this context, it is convenient to view a continuum $X \subset \mathbb{R}^2$ as a (proper) subset of the Riemann sphere $\mathbb{S}^2$ (the one-point compactification $\mathbb{R}^2$).  Given a component $U$ of $\mathbb{S}^2 \smallsetminus X$, we consider the circle of prime ends on $\partial U \subseteq X$.  A point $x \in X$ is accessible (from $U$) if and only if there exists a prime end whose principal set is $\{x\}$.  Equivalently, $x$ is accessible (from $U$) if and only if there is an \emph{external ray} in $U$ which lands on $x$; that is, for a given conformal isomorphism $\varphi$ from the open unit disk $\mathbb{D}^2 \subset \mathbb{C}$ to $U$, there exists an angle $\theta \in [0,2\pi)$ for which
\[ \lim_{r \to 1^-} \varphi(re^{i \theta}) = x .\]

Call a space \emph{planar} if it can be embedded in $\mathbb{R}^2$.  Continua $X \subsetneq \mathbb{S}^2$ which are arguably the most amenable to study via prime end theory are those whose complement $U = \mathbb{S}^2 \smallsetminus X$ is connected, and for which $\partial U = X$.  Continua of this type are precisely the planar \emph{tree-like} continua (c.f.\ \cite{bing1951}), meaning those which are homeomorphic to an inverse limit of trees.  Simplest among all tree-like continua are \emph{arc-like} continua, those which are homeomorphic to an inverse limit of arcs.  Notable examples of arc-like continua include the arc $[0,1]$ itself, the $\sin(\frac{1}{x})$-continuum, the Knaster buckethandle continuum, and the pseudo-arc.  Though there are plenty of non-planar tree-like continua, Bing \cite{bing1951} proved in 1951 that all arc-like continua are embeddable in $\mathbb{R}^2$.

Given a planar continuum $X$, in any particular embedding of $X$ in $\mathbb{R}^2$ there may be points which are inaccessible.  For example, Mazurkiewicz \cite{mazurkiewicz1929} proved in 1929 that if $X \subset \mathbb{R}^2$ is an \emph{indecomposable} continuum, i.e.\ a continuum which is not the union of two of its proper subcontinua, then the set of points in $X$ which are accessible is meager.  This naturally leads to a general question: given a planar continuum $X$ and a point $x \in X$, does there exist an embedding of $X$ in $\mathbb{R}^2$ for which $x$ is accessible, or is there some intrinsic topological obstruction in $X$ preventing $x$ from being accessible in any embedding?

A fundamental instance of the above general question arose from work of Nadler and Quinn \cite{nadler-quinn1972, nadler-quinn1973} about compactifications of half-rays and lines.  The problem was formally stated in \cite[p.229]{nadler1972}, and, following \cite{lewis1983} and \cite{problems2002}, we refer to it as the \emph{Nadler-Quinn problem}.  It asks: Given an arc-like continuum $X$ and a point $x \in X$, does there exist an embedding of $X$ in $\mathbb{R}^2$ for which $x$ is accessible?  This problem has appeared in various compilations of continuum theory problems, including as Problem~140 in \cite{lewis1983}, as Question~16 (by J.C.\ Mayer, for indecomposable continua) in \cite{problems2002}, and as the first Question (by H.\ Bruin, J.\ \v{C}in\v{c} and the second author), Question~48 (by W.\ Lewis), and Question~50 (by P.\ Minc, for indecomposable continua) in \cite{problems2018}.  In general, the study of properties of plane embeddings of arc-like continua has generated significant interest over the years, see for example \cite{mazurkiewicz1929, brechner1978, lewis1981, mayer1982, mayer1983, minc-transue1992, debski-tymchatyn1993, minc1997, anusic-bruin-cinc2017, anderson-choquet1959, ozbolt2020}.  The main result of this paper is an affirmative answer to this question of Nadler and Quinn.

\begin{thm}[Main Theorem]
\label{thm:main}
For any arc-like continuum $X$, and any point $x \in X$, there is an embedding of $X$ in $\mathbb{R}^2$ for which $x$ is an accessible point.
\end{thm}

The question of Nadler and Quinn is relevant for the study of planarity of tree-like continua.  As mentioned above, there are several known examples of non-planar tree-like continua, yet there are still few tools available to establish whether a given tree-like continuum can be embedded in $\mathbb{R}^2$ or not.  
Any new such tools would be very valuable in the field.  For example, consider the Plane Fixed Point Problem, which asks whether for any continuum $X \subset \mathbb{R}^2$ for which $\mathbb{R}^2 \smallsetminus X$ is connected, does $X$ have the \emph{fixed-point property}, i.e.\ does every continuous function of $X$ to itself leave at least one point fixed?  This is an old and central open question in continuum theory dating back to 1930 (or earlier, c.f.\ \cite{ayres1930}), and has been called ``the most interesting outstanding problem in plane topology'' by Bing \cite{bing1969}.  There are several examples of tree-like continua which do not have the fixed-point property, the first given by Bellamy \cite{bellamy1980} in 1980, and subsequent ones in \cite{oversteegen-rogers1980}, \cite{oversteegen-rogers1982}, \cite{oversteegen-rogers1980}, \cite{fearnley-wright1993}, \cite{minc1992}, \cite{minc1996}, \cite{minc1999}, \cite{minc1999;2}, \cite{minc2000}, and \cite{gutierrez-hoehn2018} (see also \cite{fugate-mohler1977}).  For almost all of these examples, it is still not known whether they can be embedded in $\mathbb{R}^2$ or not.  If any of them were to be planar, it would immediately give a negative answer to the Plane Fixed Point Problem.

Two other classic, central open problems in continuum theory, the classification of \emph{homogeneous} tree-like continua, and the classification of \emph{hereditarily equivalent} continua, have recently been resolved for planar continua \cite{hoehn-oversteegen2016}, \cite{hoehn-oversteegen2020}, but remain open in full generality.  A space $X$ is homogeneous if for any $x,y \in X$, there exists a homeomorphism $h \colon X \to X$ such that $h(x) = y$.  A continuum $X$ is hereditarily equivalent if for any subcontinuum $Y \subseteq X$ containing more than one point, $X \approx Y$.  At present, the pseudo-arc is the only known homogeneous tree-like continuum (containing more than one point).  Each hereditarily equivalent continuum is tree-like \cite{cook1970}, and the only known examples of such continua (containing more than one point) are the arc $[0,1]$ and the pseudo-arc.  Any subsequent progress on either of these two classification problems, therefore, will necessarily concern the study of non-planar tree-like continua.

The question of Nadler and Quinn can be equivalently formulated as a problem about embeddability of tree-like continua in $\mathbb{R}^2$ as follows.  Given an arc-like continuum $X$ and a point $x \in X$, form a new continuum $X_x^\perp$ by attaching an arc $A$ to $X$ so that one endpoint of $A$ is glued to $x$, and $A$ is otherwise disjoint from $X$.  The result is a tree-like continuum, and the collection of all such continua $X_x^\perp$ is in a sense the smallest possible deviation, within the class of tree-like continua, from the family of arc-like continua.  The Nadler-Quinn problem asks whether every such continuum $X_x^\perp$ is planar.  Thus, the Nadler-Quinn problem was arguably the most basic non-trivial open problem about embeddability of tree-like continua in $\mathbb{R}^2$, and its resilience against attack underscores the challenges inherent in this area.  

The affirmative answer to the Nadler-Quinn problem implies, as a corollary (see Corollary~\ref{cor:ineq embeddings} below) that for any arc-like continuum $X$ which contains an indecomposable subcontinuum (with more than one point), there exists an uncountable collection of pairwise inequivalent embeddings of $X$ in $\mathbb{R}^2$, where two embeddings of $X$ in $\mathbb{R}^2$ are \emph{equivalent} if one can be obtained from the other by composing it with a self-homeomorphism of the plane.  This answers a question of Mayer \cite{mayer1982} from 1982.

\subsection{Recent work}
\label{sec:prev work}

This paper builds on work done in \cite{ammerlaan-anusic-hoehn2023} and \cite{ammerlaan-anusic-hoehn2024}.  See Section~\ref{sec:rad deps} below for a detailed review of the concepts and results from those works which we will use in this paper.  Here we give a brief summary of what is done in those papers.

In \cite{ammerlaan-anusic-hoehn2023}, we introduce the notion of a radial departure of an interval map.  Radial departures are used to describe, in a convenient way, a ``zig-zag'' pattern in the graph of an interval map, which, when present in the bonding maps of an inverse system $\langle [-1,1],f_i \rangle$ (see Section~\ref{sec:prelim} for definitions), entails an obstruction to embedding the inverse limit $X = \varprojlim \langle [-1,1],f_i \rangle$ in $\mathbb{R}^2$ to make a given point $x \in X$ accessible.  We show (Proposition~\ref{prop:embed 0 accessible} below) that, in the absence of this ``zig-zag'' obstruction, such an embedding can be accomplished using the Anderson-Choquet Embedding Theorem \cite{anderson-choquet1959}.  We also introduce the (radial) contour factor of an interval map, which provides a means for factoring interval maps that we use to produce alternative inverse limit representations of a given arc-like continuum $X$.  We aim to produce, in this way, an inverse limit representation of $X$ satisfying the conditions of our embedding result Proposition~\ref{prop:embed 0 accessible}.

Building on the work of \cite{ammerlaan-anusic-hoehn2023}, we further develop the theory of radial contour factorizations in \cite{ammerlaan-anusic-hoehn2024}, culminating in the main technical result of that paper, which we recall as Lemma~\ref{lem:bridged s} below.  This Lemma affords, under certain coherence conditions on the radial contour factors of bonding maps and their compositions, a decomposition of the (compositions of) bonding maps leading to an inverse limit representation satisfying the conditions of Proposition~\ref{prop:embed 0 accessible}.  As an application, we give an affirmative answer to the Nadler-Quinn problem for arc-like continua which are inverse limits of \emph{simplicial} inverse systems, i.e.\ inverse systems $\langle [-1,1],f_i \rangle$ for which there exist finite sets $\{-1,1\} \subseteq S_i \subset [-1,1]$, $i \in \mathbb{N}$, such that $f_i(S_{i+1}) \subseteq S_i$, and $f_i$ is linear on each component of $[-1,1] \smallsetminus S_{i+1}$, for each $i \in \mathbb{N}$.  In the present paper, we use Lemma~\ref{lem:bridged s} in a more delicate way, together with some approximation theory for inverse systems, to achieve the same result for an arbitrary arc-like continuum $X$.

\subsection{Overview of paper}
\label{sec:overview}

This paper is structured as follows.  After recalling some standard notions and notation in Section~\ref{sec:prelim}, we review some concepts and results from \cite{ammerlaan-anusic-hoehn2023} and \cite{ammerlaan-anusic-hoehn2024} about (radial) departures and (radial) contour factors in Section~\ref{sec:rad deps}, which will be relevant for the present paper.  In Section~\ref{sec:trunc}, we introduce the notion of a truncation of an interval map relative to a finite set, which produces approximations to interval maps, with limitations on the values of their (radial) contour points.  We further develop the theory of truncations in Section~\ref{sec:mesh trunc}, and establish some technical results which will be needed for the construction in the proof of the main theorem.  In Section~\ref{sec:approx bonding maps} we state a special case of a result of Mioduszewski, which provides a tool for giving alternative inverse limit representations of arc-like continua.  Then, we turn to the proof of Theorem~\ref{thm:main} in Section~\ref{sec:main proof}.  Finally, in Section~\ref{sec:discussion}, we state some corollaries of Theorem~\ref{thm:main}, and pose some questions.

\subsection{Preliminaries}
\label{sec:prelim}

By a \emph{map} we mean a continuous function.  A map $f \colon [-1,1] \to [-1,1]$ is called \emph{piecewise-linear} if there exists $n \in \mathbb{N}$ and $-1 = c_0 < c_1 < \cdots < c_n = 1$ such that $f$ is linear on $[c_{i-1},c_{i}]$ for each $i \in \{1,\ldots,n\}$. If, additionally, $f$ is non-constant on each $[c_{i-1},c_i]$, we say that it is \emph{nowhere constant}.

A \emph{continuum} is a non-empty, compact, connected, metric space.  A continuum is \emph{degenerate} if it contains only one point, and \emph{non-degenerate} otherwise.  For a sequence $X_i$, $i \geq 1$, of continua, and a sequence $f_i \colon X_{i+1} \to X_{i}$, $i \geq 1$, of maps, the sequence
\[ \langle X_i,f_i \rangle_{i \geq 1} = \langle X_1,f_1,X_2,f_2,X_3,\ldots \rangle \]
is called an \emph{inverse system}.  The spaces $X_i$ are called \emph{factor spaces}, and the maps $f_i$ are called \emph{bonding maps}.  The \emph{inverse limit} of an inverse system $\langle X_i,f_i \rangle_{i \geq 1}$ is the space
\[ \varprojlim \left \langle X_i,f_i \right \rangle_{i \geq 1} = \{\langle x_1,x_2,\ldots \rangle: f_i(x_{i+1}) = x_i \textrm{ for each } i \geq 1\} \subseteq \prod_{i \geq 1} X_i ,\]
equipped with the subspace topology inherited from the product topology on $\prod_{i \geq 1} X_i$.  It is also a continuum.  When the index set is clear, we will use a shortened notation $\varprojlim \left \langle X_i,f_i \right \rangle$.  Note that the inverse limit representation of a continuum is not unique, for example:
\begin{itemize}
\item (Dropping finitely many coordinates) If $n_0 \geq 1$, then $\varprojlim \left \langle X_i,f_i \right \rangle_{i \geq 1} \approx \varprojlim \left \langle X_i,f_i \right \rangle_{i \geq n_0}$.
\item (Composing bonding maps) If $1 = i_1 < i_2 < i_3 < \ldots$ is an increasing sequence of integers, then $\varprojlim \left \langle X_i,f_i \right \rangle_{i \geq 1} \approx \varprojlim \left \langle X_{i_k},f_{i_k} \circ \cdots \circ f_{i_{k+1}-1} \right \rangle_{k \geq 1}$.
\end{itemize}
More generally, a result of Mioduszewski \cite{mioduszewski1963} gives necessary and sufficient conditions for two inverse limits (of polyhedra) to be homeomorphic.  We review this result in Section~\ref{sec:approx bonding maps} below, and state a special case of it for our use in the present paper (Theorem~\ref{thm:mioduszewski}).

A continuum $X$ is \emph{arc-like} if $X \approx \varprojlim \left \langle [-1,1],f_i \right \rangle$ for some sequence of maps $f_i \colon [-1,1] \to [-1,1]$, $i \in \mathbb{N}$.  It is well-known (see e.g.\ \cite{brown1960}) that an arc-like continuum is homeomorphic to the inverse limit of an inverse system $\varprojlim \left \langle [-1,1],f_i \right \rangle$ for which the bonding maps $f_i$ are nowhere constant piecewise-linear maps.  Arc-like continua are also commonly called \emph{chainable} continua, since they can be characterized as continua which admit chain covers of arbitrary small mesh.  A \emph{chain cover} of $X$ is a finite open cover $\mathcal{C} = \langle U_1,\ldots, U_n \rangle$ of $X$ indexed so that $U_i \cap U_j \neq \emptyset$ if and only if $|i-j| \leq 1$, and the \emph{mesh} of $\mathcal{C}$ is $\max\{\diam(U_i): 1 \leq i \leq n\}$.

\section{Radial departures and radial contour factors}
\label{sec:rad deps}

In this section, we recall the definitions of radial departures and radial contour factorizations from \cite{ammerlaan-anusic-hoehn2023}, and some of the results about them from \cite{ammerlaan-anusic-hoehn2023} and \cite{ammerlaan-anusic-hoehn2024} that we will make use of in this paper.  We use letters for these results, to distinguish them from the new results contained in this paper, which are labelled with numbers.  Some of the concepts in this section will pertain to functions defined on $[-1,1]$, and some will be for functions defined on $[0,1]$.  We begin with radial departures.

\begin{defn*}[Definition~4.2 of \cite{ammerlaan-anusic-hoehn2023}]
Let $f \colon [-1,1] \to [-1,1]$ be a map with $f(0) = 0$.  A \emph{radial departure} of $f$ is a pair $\langle x_1,x_2 \rangle$ such that $-1 \leq x_1 < 0 < x_2 \leq 1$ and either:
\begin{enumerate}[label=(\arabic{*})]
\item \label{pos dep} $f((x_1,x_2)) = (f(x_1),f(x_2))$; or
\item \label{neg dep} $f((x_1,x_2)) = (f(x_2),f(x_1))$.
\end{enumerate}
We say a radial departure $\langle x_1,x_2 \rangle$ of $f$ is \emph{positively oriented} (or a \emph{positive radial departure}) if \ref{pos dep} holds, and $\langle x_1,x_2 \rangle$ is \emph{negatively oriented} (or a \emph{negative radial departure}) if \ref{neg dep} holds.
\end{defn*}

See Figure~\ref{fig:prev work} (third picture) for an illustration of a map with two radial departures shown.

\begin{figure}
\begin{center}
\includegraphics{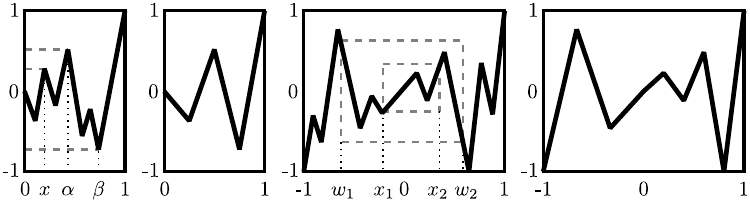}
\end{center}

\caption{Examples to illustrate the core concepts from Section~\ref{sec:rad deps}.  First picture: a map $f \colon [0,1] \to [-1,1]$, with a positive departure $x$, a positive contour point $\alpha$, and a negative contour point $\beta$ displayed.  Second picture: the contour factor of $f$.  Third picture: a map $g \colon [-1,1] \to [-1,1]$ with a positive radial departure $\langle x_1,x_2 \rangle$ and a negative radial departure $\langle w_1,w_2 \rangle$ displayed.  Fourth picture: the radial contour factor of $g$.}
\label{fig:prev work}
\end{figure}

When a map has both positive and negative radial departures, as with the map depicted in Figure~\ref{fig:prev work} (third picture), it implies the presence of a ``zig-zag'' pattern in the graph of the map centered around $x=0$.  The next Proposition states that if, in an inverse system $\langle [-1,1],f_i \rangle$ with $f_i(0) = 0$ for each $i \in \mathbb{N}$, none of the maps $f_i$ has such a ``zig-zag'' shape in its graph, then the inverse limit $X = \varprojlim \langle [-1,1],f_i \rangle$ can be embedded in $\mathbb{R}^2$ in such a way as to make the point $\langle 0,0,\ldots \rangle \in X$ accessible.  This is accomplished using the Anderson-Choquet Embedding Theorem \cite{anderson-choquet1959}.  The key idea is, roughly, that if all radial departures of a map have the same orientation, then the left half of its graph can be ``swung around'' to the right side, preserving (approximately) the $y$-values at all points and without causing any self-intersections, thus exposing the point $0$ from the left.

\begin{propalph}[Proposition~5.2 of \cite{ammerlaan-anusic-hoehn2023}]
\label{prop:embed 0 accessible}
Let $f_i \colon [-1,1] \to [-1,1]$, $i \in \mathbb{N}$, be maps with $f_i(0) = 0$ for each $n$.  Suppose that for each $i$, all radial departures of $f_i$ have the same orientation.  Then there exists an embedding of $X = \varprojlim \left \langle [-1,1],f_i \right \rangle$ into $\mathbb{R}^2$ for which the point $\langle 0,0,\ldots \rangle \in X$ is accessible.
\end{propalph}

The next Proposition gives a characterization of radial departures of a composition of two maps in terms of the radial departures of the individual maps.

\begin{propalph}[Proposition~4.6 of \cite{ammerlaan-anusic-hoehn2023}]
\label{prop:comp dep}
Let $f,g \colon [-1,1] \to [-1,1]$ be maps with $f(0) = g(0) = 0$, and let $x_1,x_2$ be such that $-1 \leq x_1 < 0 < x_2 \leq 1$.  Then $\langle x_1,x_2 \rangle$ is a positive (respectively, negative) radial departure of $f \circ g$ if and only if either:
\begin{enumerate}
\item $\langle x_1,x_2 \rangle$ is a positive radial departure of $g$ and $\langle g(x_1),g(x_2) \rangle$ is a positive (respectively, negative) radial departure of $f$; or
\item $\langle x_1,x_2 \rangle$ is a negative radial departure of $g$ and $\langle g(x_2),g(x_1) \rangle$ is a negative (respectively, positive) radial departure of $f$.
\end{enumerate}
\end{propalph}

The second key concept from \cite{ammerlaan-anusic-hoehn2023} we will need is the (radial) contour factor of an interval map.  With the radial contour factor $t_f$ of a map $f \colon [-1,1] \to [-1,1]$ with $f(0) = 0$, we are able to factor $f$ as a composition $f = t_f \circ s$ for some map $s \colon [-1,1] \to [-1,1]$.  This provides us with a tool for constructing, by factoring (compositions of) bonding maps, alternative inverse limit representations of a given arc-like continuum.  In particular, this will help us construct one which satisfies the conditions of Proposition~\ref{prop:embed 0 accessible} above.

\begin{defn*}[Definition~3.1 of \cite{ammerlaan-anusic-hoehn2023}]
Let $f \colon [0,1] \to [-1,1]$ be a map with $f(0) = 0$.  A \emph{departure} of $f$ is a number $x > 0$ such that $f(x) \notin f([0,x))$.  We say a departure $x$ of $f$ is \emph{positively oriented} (or a \emph{positive departure}) if $f(x) > 0$, and $x$ is \emph{negatively oriented} (or a \emph{negative departure}) if $f(x) < 0$.

A \emph{contour point} of $f$ is a departure $\alpha$ such that for any departure $x$ of $f$ with $x > \alpha$, there exists a departure $y$ of $f$, of orientation opposite to that of $\alpha$, such that $\alpha < y \leq x$.
\end{defn*}

For convenience in statements and arguments below, we may also include $\alpha = 0$ as a contour point of $f$.

\begin{defn*}[Definition~3.2 of \cite{ammerlaan-anusic-hoehn2023}]
Let $f \colon [0,1] \to [-1,1]$ be a non-constant piecewise-linear map with $f(0) = 0$.  The \emph{contour factor} of $f$ is the piecewise-linear map $t_f \colon [0,1] \to [-1,1]$ defined as follows:

Let $0 = \alpha_0 < \alpha_1 < \alpha_2 < \cdots < \alpha_n$ be the contour points of $f$.  Then $t_f \left( \frac{i}{n} \right) = f(\alpha_i)$ for each $i = 0,\ldots,n$, and $t_f$ is defined to be linear in between these points.
\end{defn*}

See Figure~\ref{fig:prev work} (first picture) for an illustration of a map with some departures and contour points shown, and (second picture) the contour factor of that map.

We remark that if $i \in \{0,\ldots,n-1\}$, then $f([0,\alpha_{i+1}]) = f([\alpha_i,\alpha_{i+1}])$.  In fact, if $\alpha_{i+1}$ is a positive contour point (so that $\alpha_i$ is a negative contour point or $0$), then $f([0,\alpha_{i+1})) = [f(\alpha_i),f(\alpha_{i+1}))$, and similarly if $\alpha_{i+1}$ is a negative contour point (so that $\alpha_i$ is a positive contour point or $0$), then $f([0,\alpha_{i+1})) = (f(\alpha_{i+1}),f(\alpha_i)]$.

\begin{lemalph}[Lemma~3.4 from \cite{ammerlaan-anusic-hoehn2023}]
\label{lem:same contour}
Let $f,g \colon [0,1] \to [-1,1]$ be non-constant piecewise-linear maps with $f(0) = g(0) = 0$.  Then $t_f = t_g$ if and only if:
\begin{enumerate}
\item for each departure $x$ of $f$, there exists a departure $x'$ of $g$ such that $f([0,x)) = g([0,x'))$; and
\item for each departure $x'$ of $g$, there exists a departure $x$ of $f$ such that $f([0,x)) = g([0,x'))$.
\end{enumerate}
\end{lemalph}

We add the following variant of Lemma~\ref{lem:same contour}, in which the conditions (1) and (2) are weakened slightly.  As such, this variant can be easier to use to establish that two maps have the same contour factor.

\begin{lem}
\label{lem:same contour 2}
Let $f,g \colon [0,1] \to [-1,1]$ be non-constant piecewise-linear maps with $f(0) = g(0) = 0$.  Then $t_f = t_g$ if and only if:
\begin{enumerate}
\item for each contour point $\alpha$ of $f$, there exists a departure $x'$ of $g$ such that $f([0,\alpha)) = g([0,x'))$; and
\item for each contour point $\alpha'$ of $g$, there exists a departure $x$ of $f$ such that $f([0,x)) = g([0,\alpha'))$.
\end{enumerate}
\end{lem}

\begin{proof}
The forward implication follows from Lemma~\ref{lem:same contour}.  For the converse, suppose that (1) and (2) of the present Lemma hold.  We prove that (1) of Lemma~\ref{lem:same contour} follows.

Let $x$ be a departure of $f$, and suppose $f(x) > 0$ (the case where $f(x) < 0$ is similar).  Let $0 = \alpha_0 < \alpha_1 < \alpha_2 < \cdots < \alpha_n$ be the contour points of $f$, and let $i \in \{0,\ldots,n-1\}$ be such that $x \in (\alpha_i,\alpha_{i+1}]$.  By assumption (1), there exist departures $x_i',x_{i+1}'$ of $g$ such that $g([0,x_i')) = f([0,\alpha_i))$ and $g([0,x_{i+1}')) = f([0,\alpha_{i+1}))$.  Now $f(x) \notin f([0,\alpha_i)) = g([0,x_i'))$ but $f(x) \in f([0,\alpha_{i+1})) = g([0,x_{i+1}'))$, so there must be a departure $x' \in (x_i',x_{i+1}']$ of $g$ such that $g(x') = f(x)$.  Note that $f([0,x)) = [f(\alpha_i),f(x)) \subseteq g([0,x'))$.  Furthermore, $g(y) < g(x') = f(x)$ for all $y \in [0,x')$, and $g([0,x')) \subseteq g([0,x_{i+1}')) = f([0,\alpha_{i+1})) = [f(\alpha_i),f(\alpha_{i+1}))$, meaning that $g(y) \geq f(\alpha_i)$ for all $y \in [0,x')$.  It follows that $g([0,x')) \subseteq f([0,x))$.  Therefore, $f([0,x)) = g([0,x'))$.

Similarly, we can prove that (2) of Lemma~\ref{lem:same contour} holds.  Therefore, by Lemma~\ref{lem:same contour}, $t_f = t_g$.
\end{proof}

For the next two Definitions, let $\mathsf{r} \colon [0,1] \to [-1,0]$ be the function $\mathsf{r}(x) = -x$

\begin{defn*}[Definition~4.1 of \cite{ammerlaan-anusic-hoehn2023}]
Let $f \colon [-1,1] \to [-1,1]$ be a map with $f(0) = 0$.
\begin{itemize}
\item A \emph{right departure} of $f$ is a number $x > 0$ such that $x$ is a departure of $f {\restriction}_{[0,1]}$.
\item A \emph{left departure} of $f$ is a number $x < 0$ such that $-x$ is a departure of $f {\restriction}_{[-1,0]} \circ \mathsf{r}$; i.e.\ a number $x < 0$ such that $f(x) \notin f((x,0])$.
\end{itemize}
If $x$ is either a right departure or a left departure of $f$, then we say $x$ is \emph{positively oriented} if $f(x) > 0$, and $x$ is \emph{negatively oriented} if $f(x) < 0$.
\begin{itemize}
\item A \emph{right contour point} of $f$ is a number $\alpha > 0$ such that $\alpha$ is a contour point of $f {\restriction}_{[0,1]}$.
\item A \emph{left contour point} of $f$ is a number $\beta < 0$ such that $-\beta$ is a contour point of $f {\restriction}_{[-1,0]} \circ \mathsf{r}$.
\end{itemize}
\end{defn*}

\begin{defn*}[Definition~6.1 of \cite{ammerlaan-anusic-hoehn2023}]
Let $f \colon [-1,1] \to [-1,1]$ be a piecewise-linear map with $f(0) = 0$, and such that both restrictions $f {\restriction}_{[0,1]}$ and $f {\restriction}_{[-1,0]}$ are non-constant.  The \emph{radial contour factor} of $f$ is the piecewise-linear map $t_f \colon [-1,1] \to [-1,1]$ such that:
\begin{enumerate}
\item $t_f {\restriction}_{[0,1]}$ is the contour factor of $f {\restriction}_{[0,1]}$; and
\item $t_f {\restriction}_{[-1,0]} \circ \mathsf{r}$ is the contour factor of $f {\restriction}_{[-1,0]} \circ \mathsf{r}$.
\end{enumerate}
\end{defn*}

See Figure~\ref{fig:prev work} (fourth picture) for an illustration of the radial contour factor of a map.

For brevity, we will say that a map $f \colon [-1,1] \to [-1,1]$ \emph{has well-defined radial contour factor} if $f$ is piecewise-linear, $f(0) = 0$, and both restrictions $f {\restriction}_{[0,1]}$ and $f {\restriction}_{[-1,0]}$ are non-constant.

\begin{defn*}[Definition~6.3 from \cite{ammerlaan-anusic-hoehn2023}]
Let $f,g \colon [-1,1] \to [-1,1]$ be maps with $f(0) = g(0) = 0$.  We say \emph{$f$ and $g$ have the same radial departures} if for any $y_1,y_2 \in [-1,1]$, there exists a radial departure $\langle x_1,x_2 \rangle$ of $f$ with $f(x_1) = y_1$ and $f(x_2) = y_2$ if and only if there exists a radial departure $\langle x_1',x_2' \rangle$ of $g$ with $g(x_1') = y_1$ and $g(x_2') = y_2$.
\end{defn*}

\begin{lemalph}[Lemma~6.4 from \cite{ammerlaan-anusic-hoehn2023}]
\label{lem:f tf same dep}
Let $f \colon [-1,1] \to [-1,1]$ be a map with well-defined radial contour factor $t_f$.  Then $f$ and $t_f$ have the same radial departures.
\end{lemalph}

We add one simple result which will be helpful later in the paper for proving that all radial departures of certain compositions of maps have the same orientation.

\begin{lem}
\label{lem:comp same dep}
Let $f,g_1,g_2 \colon [-1,1] \to [-1,1]$ be maps with $f(0) = g_1(0) = g_2(0) = 0$.  If $g_1$ and $g_2$ have the same radial departures, then $f \circ g_1$ and $f \circ g_2$ have the same radial departures.
\end{lem}

\begin{proof}
Suppose that $g_1$ and $g_2$ have the same radial departures.  We will prove that if $z_1,z_2 \in [-1,1]$ are such that there is a radial departure of $f \circ g_1$ mapping to $z_1,z_2$ under $f \circ g_1$, then there is a radial departure of $f \circ g_2$ mapping to $z_1,z_2$ under $f \circ g_2$.  Since $g_1$ and $g_2$ are interchangeable here, this will suffice to prove the Lemma.

Let $\langle x_1,x_2 \rangle$ be a radial departure of $f \circ g_1$, let $y_1 = g_1(x_1)$ and $y_2 = g_1(x_2)$, and let $z_1 = f(y_1) = f \circ g_1(x_1)$ and $z_2 = f(y_2) = f \circ g_1(x_2)$.  By Proposition~\ref{prop:comp dep}, $\langle x_1,x_2 \rangle$ is a radial departure of $g_1$ and $\langle y_1,y_2 \rangle$ is a radial departure of $f$.  Now, since $g_1$ and $g_2$ have the same radial departures, there exists a radial departure $\langle x_1',x_2' \rangle$ of $g_2$ such that $g_2(x_1') = y_1$ and $g_2(x_2') = y_2$.  By Proposition~\ref{prop:comp dep}, it follows that $\langle x_1',x_2' \rangle$ is a radial departure of $f \circ g_2$, and $f \circ g_2(x_1') = f(y_1) = z_1$ and $f \circ g_2(x_2') = f(y_2) = z_2$.
\end{proof}

The final result in this section, Lemma~\ref{lem:bridged s} below, is the main technical result from \cite{ammerlaan-anusic-hoehn2024}.  It states that, given three maps $f_1,f_2,f_3 \colon [-1,1] \to [-1,1]$, if the radial contour factors are ``stable'' under composition (that is, if the radial contour factors of $f_1$ and of $f_1 \circ f_2$ are the same, and the same goes for $f_2$ and $f_2 \circ f_3$), then by a careful choice of map $s$ in the factorization $f_1 \circ f_2 = t_{f_1} \circ s$, we may arrange that the composition $s \circ t_{f_3}$ has no negative radial departures.

\begin{lemalph}[Lemma~4.1 from \cite{ammerlaan-anusic-hoehn2024}]
\label{lem:bridged s}
Let $f_1,f_2,f_3 \colon [-1,1] \to [-1,1]$ be maps with well-defined radial contour factors.  Suppose that $t_{f_1} = t_{f_1 \circ f_2}$ and $t_{f_2} = t_{f_2 \circ f_3}$.  Then there exists $s \colon [-1,1] \to [-1,1]$ with $s(0) = 0$ such that:
\begin{enumerate}
\item $t_{f_1} \circ s = f_1 \circ f_2$; and
\item $s \circ t_{f_3}$ has no negative radial departures.
\end{enumerate}
\end{lemalph}

If the conditions of Lemma~\ref{lem:bridged s} are satisfied by each triple of consecutive maps in an inverse system $\langle [-1,1],f_i \rangle$, then we can use the Lemma to change to a new inverse system, whose inverse limit is homeomorphic to the original, and which satisfies the conditions of Proposition~\ref{prop:embed 0 accessible} above.  In that way, we were able to solve the Nadler-Quinn problem for arc-like continua which are inverse limits of simplicial inverse systems in \cite{ammerlaan-anusic-hoehn2024}.  To treat the Nadler-Quinn problem in full generality, we will make a less straight-forward application of Lemma~\ref{lem:bridged s}, not on the original inverse system but on one derived from it using some approximation theory, which we develop in the next sections.

\section{Truncations of interval maps}
\label{sec:trunc}

This section is devoted to the notion of the truncation of an interval map $f$ with respect to a finite set $V$ (Definition~\ref{defn:trunc} below).  This is an elementary notion, which may be of interest in its own right, as a way of partially ``denoising'' the function $f$.  For the purposes of this paper, the key property is that truncation produces an approximation to $f$, all of whose contour points take values in the finite set $V$ (see Lemma~\ref{lem:deps trunc} below).

In the following Definition and Lemma, let $[a,b]$ represent either $[0,1]$ or $[-1,1]$.

\begin{defn}
\label{defn:trunc}
Let $f \colon [a,b] \to [-1,1]$ be a piecewise-linear map, and let $V \subset [-1,1]$ be a finite set such that $f([a,b]) \cap V \neq \emptyset$.  We define the \emph{truncation} of $f$ with respect to $V$, denoted $\trunc{f}{V}$, as follows.  Given $x \in [a,b] \smallsetminus f^{-1}(V)$, let $C_x = C_x^{f,V}$ be the connected component of $[a,b] \smallsetminus f^{-1}(V)$ which contains $x$.  By $\partial C_x$ we mean the boundary of $C_x$ in the space $[a,b]$ (as opposed to in $\mathbb{R}$).  Define $\trunc{f}{V} \colon [a,b] \to [-1,1]$ by:
\[ \trunc{f}{V}(x) = \begin{cases}
f(x) & \textrm{if } x \in f^{-1}(V), \textrm{ or if } x \notin f^{-1}(V) \textrm{ and} \\
& \;\; f(\partial C_x) \textrm{ has two elements} \\
v & \textrm{if } x \notin f^{-1}(V) \textrm{ and } f(\partial C_x) = \{v\}.
\end{cases} \]
\end{defn}

\begin{figure}
\begin{center}
\includegraphics{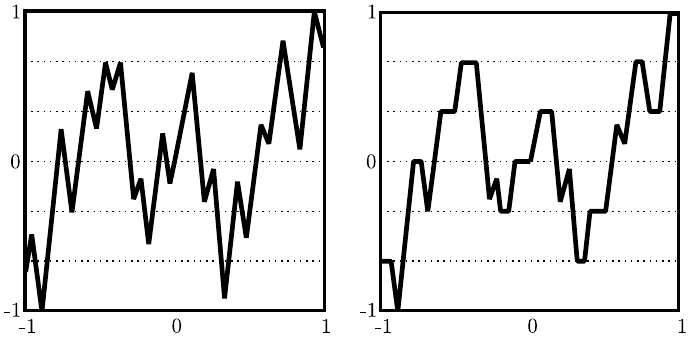}
\end{center}

\caption{On the left, a piecewise linear map $f \colon [-1,1] \to [-1,1]$, with a finite set $V \subset [-1,1]$ whose elements are indicated by dotted horizontal lines (and also $-1,1 \in V$).  On the right, the truncation $\trunc{f}{V}$ of $f$ with respect to $V$.}
\label{fig:trunc}
\end{figure}

See Figure~\ref{fig:trunc} for an illustration of a truncation of a map.  Note the following:
\begin{itemize}
\item $\trunc{f}{V}(x) = f(x)$ for all $x$ such that either $f(x) \in V$ or $\trunc{f}{V}(x) \notin V$;
\item $\trunc{f}{V}$ is continuous and piecewise-linear; and
\item if $0 \in V$ and $f(0) = 0$, then $\trunc{f}{V}(0) = 0$.
\end{itemize}

\begin{lem}
\label{lem:interval image trunc}
Let $f \colon [a,b] \to [-1,1]$ be a piecewise-linear map, and let $V \subset [-1,1]$ be a finite set such that $f([a,b]) \cap V \neq \emptyset$.  Then for any interval $A \subseteq [a,b]$, either $\trunc{f}{V}(A)$ is degenerate, or $\trunc{f}{V}(A) \subseteq f(A)$.
\end{lem}

\begin{proof}
Let $A \subseteq [a,b]$ be an interval.  Suppose $x \in A$ and $\trunc{f}{V}(x) \notin f(A)$.  So $\trunc{f}{V}(x) \neq f(x)$, which implies $x \notin f^{-1}(V)$ and $f(\partial C_x) = \{v\}$ for some $v \in V$, and $\trunc{f}{V}(x) = v$.  Since $v \notin f(A)$, we have that $A \subseteq C_x$.  Now for any $y \in A$, we have $C_y = C_x$, and so $\trunc{f}{V}(y) = v$, thus $\trunc{f}{V}(A) = \{v\}$ is degenerate.
\end{proof}

The next Lemma gives a characterization of departures of $\trunc{f}{V}$ in terms of departures of $f$.

\begin{lem}
\label{lem:deps trunc}
Let $f \colon [0,1] \to [-1,1]$ be a piecewise-linear map such that $f(0) = 0$, and let $V \subset [-1,1]$ be a finite set with $0 \in V$.  A point $x \in (0,1]$ is a departure of $\trunc{f}{V}$ if and only if either:
\begin{enumerate}[label=(\alph{*})]
\item $f(x) \in V$ and $x$ is a departure of $f$, or
\item $f(x) \notin V$, $\partial C_x = \{c_1,c_2\}$ for some $c_1,c_2 \in [0,1]$ with $c_1 < c_2$, $f(x) \notin f([c_1,x))$, and $c_2$ is a departure of $f$.
\end{enumerate}
Moreover, if $x$ is a departure of $\trunc{f}{V}$, then $\trunc{f}{V}(x) = f(x)$, and
\begin{itemize}
\item if $f(x) > 0$, then $\trunc{f}{V}([0,x)) = [m,f(x))$, where
\[ m = \min \left( f([0,x)) \cap V \right) ;\]
\item if $f(x) < 0$, then $\trunc{f}{V}([0,x)) = (f(x),M]$, where
\[ M = \max \left( f([0,x)) \cap V \right) .\]
\end{itemize}
\end{lem}

\begin{proof}
Suppose that $x$ is a departure of $\trunc{f}{V}$.  Suppose that $f(x) \in V$, which means $\trunc{f}{V}(x) = f(x)$.  Then $x$ must be a departure of $f$, since if there were some $x' < x$ with $f(x') = f(x)$, then $\trunc{f}{V}(x') = f(x)$ as well, which would contradict the assumption that $x$ is a departure of $\trunc{f}{V}$.  Thus (a) holds.  Suppose now that $f(x) \notin V$.  Consider the component $C_x$ of $x$ in $[0,1] \smallsetminus f^{-1}(V)$.  If $f(\partial C_x) = \{v\}$ for some $v \in V$, then $\trunc{f}{V}(\overline{C_x}) = \{v\}$, contradicting the assumption that $x$ is a departure of $\trunc{f}{V}$.  So we must have $\partial C_x = \{c_1,c_2\}$ for some $c_1 < c_2$, and $f(c_1) \neq f(c_2)$, which means $\trunc{f}{V}(y) = f(y)$ for each $y \in [c_1,c_2]$.  Thus, since $x$ is a departure of $\trunc{f}{V}$, we have $f(x) = \trunc{f}{V}(x) \notin \trunc{f}{V}([0,x)) \supseteq \trunc{f}{V}([c_1,x)) = f([c_1,x))$.  Finally, if $c_2$ were not a departure of $f$, i.e.\ if $f(c_2) \in f([0,c_2))$, then we would have $f(c_2) \in f([0,c_1))$, i.e.\ $f(x') = f(c_2)$ for some $x' \in [0,c_1)$.  But then $\trunc{f}{V}(x') = f(x') = f(c_2)$, and $\trunc{f}{V}(c_1) = f(c_1)$, so $f(c_1),f(c_2) \in \trunc{f}{V}([0,x))$, meaning that the interval between $f(c_1)$ and $f(c_2)$ is contained in $\trunc{f}{V}([0,x))$.  In particular, $\trunc{f}{V}(x) = f(x) \in \trunc{f}{V}([0,x))$, contradicting the assumption that $x$ is a departure of $\trunc{f}{V}$.  Thus (b) holds.

For the converse, first suppose (a).  So $\trunc{f}{V}(x) = f(x)$, and $f(x) \notin f([0,x))$.  Note that we cannot have that $\trunc{f}{V}([0,x))$ is degenerate, since this would mean $\trunc{f}{V}([0,x)) = \{0\}$ and so $f(x) = \trunc{f}{V}(x) = 0$, contradicting the assumption that $x$ is a departure of $f$.  Therefore, by Lemma~\ref{lem:interval image trunc}, we have $\trunc{f}{V}([0,x)) \subseteq f([0,x))$, hence $\trunc{f}{V}(x) = f(x) \notin \trunc{f}{V}([0,x))$, meaning $x$ is a departure of $\trunc{f}{V}$.  Now suppose (b).  Then $\trunc{f}{V}(y) = f(y)$ for all $y \in [c_1,c_2]$, which means
\[ \trunc{f}{V}(x) = f(x) \notin f([c_1,x)) = \trunc{f}{V}([c_1,x)) .\]
Suppose for a contradiction that $\trunc{f}{V}(x') = \trunc{f}{V}(x)$ for some $x' < c_1$.  So $\trunc{f}{V}(x') = \trunc{f}{V}(x) = f(x)$, which is between $f(c_1)$ and $f(c_2)$.  Consider the component $C_{x'}$ of $x'$ in $[0,1] \smallsetminus f^{-1}(V)$.  Since $C_{x'} \subseteq [0,c_1)$, and since $c_2$ is a departure of $f$, it follows that $f(\partial C_{x'}) = \{f(c_1)\}$.  But then $\trunc{f}{V}(x') = f(c_1) \neq f(x)$, a contradiction.  Therefore,
\[ \trunc{f}{V}(x) \notin \trunc{f}{V}([0,c_1)) \]
as well, and so we have that $x$ is a departure of $\trunc{f}{V}$.

For the last statement, suppose that $x$ is a departure of $\trunc{f}{V}$, and that $f(x) > 0$ (the case $f(x) < 0$ is similar).  It is clear in each of the cases (a) and (b) above that $\trunc{f}{V}(x) = f(x)$.  Let $m = \min \left( f([0,x)) \cap V \right)$.  Note that $m \in \trunc{f}{V}([0,x))$ as well, since $\trunc{f}{V}$ and $f$ agree on $f^{-1}(V)$.  So $\trunc{f}{V}([0,x)) \supseteq [m,f(x))$.  For the reverse inclusion, since $x$ is a departure of $\trunc{f}{V}$ we have $f(x') < f(x)$ for all $x' \in [0,x)$.  Let $x' \in [0,x)$ with $f(x') \leq 0$.  If $f(x') \in V$, then $\trunc{f}{V}(x') = f(x') \geq m$ by minimality of $m$.  If $f(x') \notin V$, then consider the component $C_{x'}$ of $x'$ in $[0,1] \smallsetminus f^{-1}(V)$.  Note that $C_{x'} \subset [0,x)$ since $f(x') < 0$.  If $v \in f(\partial C_{x'})$, then $v \geq m$ by minimality of $m$.  It follows that $\trunc{f}{V}(x') \geq m$, hence $\trunc{f}{V}([0,x)) \subseteq [m,f(x))$.  Therefore, $\trunc{f}{V}([0,x)) = [m,f(x))$.
\end{proof}

In the case that $x$ is a contour point of $\trunc{f}{V}$, then in fact alternative (a) of Lemma~\ref{lem:deps trunc} must hold, as the next Lemma shows.

\begin{lem}
\label{lem:contour trunc}
Let $f \colon [0,1] \to [-1,1]$ be a piecewise-linear map such that $f(0) = 0$, and let $V \subset [-1,1]$ be a finite set with $0 \in V$.  If $x \in (0,1]$ is a contour point of $\trunc{f}{V}$, then $f(x) \in V$ and $x$ is a departure of $f$.
\end{lem}

\begin{proof}
Suppose $x$ is a contour point of $\trunc{f}{V}$.  By Lemma~\ref{lem:deps trunc}, if $f(x) \notin V$, then $\partial C_x = \{c_1,c_2\}$ for some $c_1,c_2 \in [0,1]$ with $c_1 < c_2$, and $c_2$ is a departure of $f$.  Then again by Lemma~\ref{lem:deps trunc}, $c_2$ is a departure of $\trunc{f}{V}$.  But clearly each departure of $\trunc{f}{V}$ in $(x,c_2]$ has the same orientation as $x$, a contradiction with the definition of a contour point.  Therefore $f(x) \in V$, and $x$ is a departure of $f$ (by Lemma~\ref{lem:deps trunc}).
\end{proof}

Let $f \colon [-1,1] \to [-1,1]$ be such that $f(0) = 0$, and let $V \subset [-1,1]$ be a finite set with $0 \in V$.  Recall that $\trunc{f}{V}$ has a well-defined radial contour factor if the restrictions $\trunc{f}{V} {\restriction}_{[0,1]}$ and $\trunc{f}{V} {\restriction}_{[-1,0]}$ are non-constant.  This is equivalent to the condition that both of the sets $f([0,1]) \cap V \smallsetminus \{0\}$ and $f([-1,0]) \cap V \smallsetminus \{0\}$ are non-empty.  A sufficient condition for $\trunc{f}{V}$ to have a well-defined radial contour factor is that the diameters of the sets $f([0,1])$ and $f([-1,0])$ are greater than or equal to $\mesh(V)$.

\begin{lem}
\label{lem:compare trunc}
Let $f,g \colon [-1,1] \to [-1,1]$ be piecewise-linear maps such that $f(0) = g(0) = 0$, and let $V,W \subset [-1,1]$ be finite sets with $0 \in V \subseteq W$ and such that $\trunc{f}{V}$ and $\trunc{g}{V}$ have well-defined radial contour factors.  Then:
\begin{enumerate}[label=(\arabic{*})]
\item \label{comp trunc} if $f^{-1}(V)$ is finite (e.g.\ if $f$ is nowhere constant), then $\trunc{f}{V} \circ \trunc{g}{f^{-1}(V)} = \trunc{f \circ g}{V}$;
\item if $f$ and $g$ have the same radial contour factors, then so do $\trunc{f}{V}$ and $\trunc{g}{V}$;
\item $\trunc{f}{V}$ and $\trunc{\trunc{f}{V}}{W}$ have the same radial contour factors;
\item $\trunc{\trunc{f}{V}}{W} = \trunc{\trunc{f}{W}}{V}$; and
\item \label{fV tfWV} $\trunc{f}{V}$ and $\trunc{t_{\trunc{f}{W}}}{V}$ have the same radial contour factors.
\end{enumerate}
\end{lem}

We leave it to the reader to confirm that each of the maps $f$, $g$, $\trunc{\trunc{f}{V}}{W}$, $\trunc{f}{W}$, and $\trunc{t_{\trunc{f}{W}}}{V}$ have well-defined radial contour factors under the assumption that $\trunc{f}{V}$ and $\trunc{g}{V}$ have well-defined radial contour factors.  See Figure~\ref{fig:successive trunc} for an illustration of a map truncated with respect to two finite sets $V \subset W$ in succession.  The reader may also find it helpful to refer to Figure~\ref{fig:refine trunc contours} below, which depicts the same map $f$ and the same sets $V,W$ as in Figure~\ref{fig:successive trunc}, but there $\trunc{f}{W}$ is drawn instead of $\trunc{\trunc{f}{V}}{W}$.

\begin{figure}
\begin{center}
\includegraphics{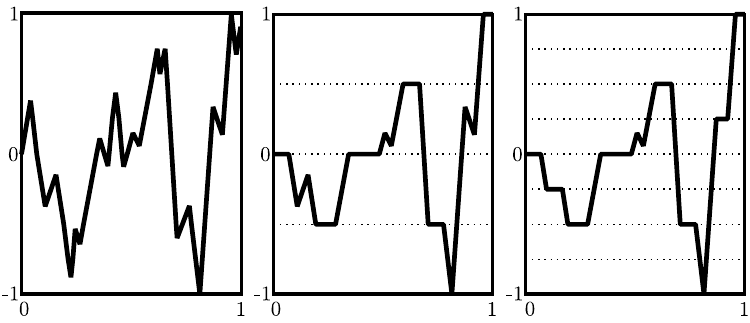}
\end{center}

\caption{On the left, a piecewise linear map $f \colon [0,1] \to [-1,1]$.  In the center, the truncation $\trunc{f}{V}$ of $f$ with respect to a finite set $V \subset [-1,1]$ whose elements are indicated by dotted horizontal lines (and also $-1,1 \in V$).  On the right, the truncation $\trunc{\trunc{f}{V}}{W}$ of $\trunc{f}{V}$ with respect to another finite set $W \supset V$.}
\label{fig:successive trunc}
\end{figure}

\begin{proof}
For \ref{comp trunc}, let $x \in [-1,1]$.  We first point out the following straightforward observation:
\begin{itemize}
\item $x \notin (f \circ g)^{-1}(V)$ if and only if $g(x) \notin f^{-1}(V)$, and in this case,
\[ C_x^{f \circ g, V} = C_x^{g, f^{-1}(V)} .\]
\end{itemize}

If $x \in (f \circ g)^{-1}(V)$, then $g(x) \in f^{-1}(V)$, and so $\trunc{g}{f^{-1}(V)}(x) = g(x)$ and $\trunc{f}{V} \circ \trunc{g}{f^{-1}(V)}(x) = \trunc{f}{V}(g(x)) = f \circ g(x) = \trunc{f \circ g}{V}(x)$.

Suppose now that $x \notin (f \circ g)^{-1}(V)$, so that $g(x) \notin f^{-1}(V)$.  Let $C_x = C_x^{f \circ g, V} = C_x^{g, f^{-1}(V)}$.  If $g(\partial C_x) = \{p\}$ for some $p \in f^{-1}(V)$, then $f \circ g(\partial C_x) = \{f(p)\}$, and we have $\trunc{g}{f^{-1}(V)}(x) = p$ and $\trunc{f}{V} \circ \trunc{g}{f^{-1}(V)}(x) = f(p) = \trunc{f \circ g}{V}(x)$.  Suppose now that $g(\partial C_x) = \{p_1,p_2\}$ for some $p_1 \neq p_2$, which means $\trunc{g}{f^{-1}(V)}(x) = g(x)$.  If $f(p_1) = f(p_2) = v$ for some $v \in V$, i.e.\ if $f \circ g(\partial C_x) = \{v\}$, then $\trunc{f}{V} \circ \trunc{g}{f^{-1}(V)}(x) = \trunc{f}{V}(g(x)) = v = \trunc{f \circ g}{V}(x)$.  Finally, if $f(p_1) \neq f(p_2)$, then $\trunc{f}{V} \circ \trunc{g}{f^{-1}(V)}(x) = f \circ g(x) = \trunc{f \circ g}{V}(x)$.

For (2), suppose $f$ and $g$ have the same radial contour factors, i.e.\ $t_f = t_g$.  To prove $\trunc{f}{V}$ and $\trunc{g}{V}$ have the same radial contour factor, we treat the ``right halves'', $\trunc{f}{V} {\restriction}_{[0,1]}$ and $\trunc{g}{V} {\restriction}_{[0,1]}$ first, using Lemma~\ref{lem:same contour 2} as follows.  Let $\alpha \in (0,1]$ be a right contour point of $\trunc{f}{V}$.  By Lemma~\ref{lem:contour trunc}, $f(\alpha) \in V$ and $\alpha$ is a right departure of $f$.  By Lemma~\ref{lem:same contour}, there exists a right departure $x' \in (0,1]$ of $g$ such that $f([0,\alpha)) = g([0,x'))$.  Now $g(x') = f(\alpha) \in V$, so by Lemma~\ref{lem:deps trunc}, $x'$ is a right departure of $\trunc{g}{V}$.  Suppose $f(\alpha) > 0$.  Let $m = \min \left( f([0,\alpha)) \cap V \right) = \min \left( g([0,x')) \cap V \right)$.  Then $\trunc{f}{V}([0,\alpha)) = [m,f(\alpha)) = [m,g(x')) = \trunc{g}{V}([0,x'))$ by Lemma~\ref{lem:deps trunc}.  The case where $f(\alpha) < 0$ is similar.  Thus, we have established condition (1) of Lemma~\ref{lem:same contour 2}.  Since $f$ and $g$ are interchangeable here, we may similarly establish condition (2) of Lemma~\ref{lem:same contour 2}.  So by Lemma~\ref{lem:same contour 2}, this completes the proof that $\trunc{f}{V} {\restriction}_{[0,1]}$ and $\trunc{g}{V} {\restriction}_{[0,1]}$ have the same contour factor.  The proof for the ``left halves'', $\trunc{f}{V} {\restriction}_{[-1,0]}$ and $\trunc{g}{V} {\restriction}_{[-1,0]}$ (i.e.\ the proof that the contour factors of $\left( \trunc{f}{V} {\restriction}_{[-1,0]} \right) \circ \mathsf{r}$ and $\left( \trunc{g}{V} {\restriction}_{[-1,0]} \right) \circ \mathsf{r}$ are the same), is similar.

For (3), as above, we treat the ``right halves'', $\trunc{f}{V} {\restriction}_{[0,1]}$ and $\trunc{\trunc{f}{V}}{W} {\restriction}_{[0,1]}$ first, using Lemma~\ref{lem:same contour 2} as follows.  Let $\alpha \in (0,1]$ be a right contour point of $\trunc{f}{V}$.  Then by Lemma~\ref{lem:contour trunc}, $\trunc{f}{V}(\alpha) = f(\alpha) \in V$, hence $\trunc{f}{V}(\alpha) \in W$, and so by Lemma~\ref{lem:deps trunc}, $\alpha$ is a right departure of $\trunc{\trunc{f}{V}}{W}$ and $\trunc{\trunc{f}{V}}{W}(\alpha) = f(\alpha)$.  Suppose $f(\alpha) > 0$.  Let $m = \min \left( f([0,\alpha)) \cap V \right)$, so that $\trunc{f}{V}([0,\alpha)) = [m,f(\alpha))$ according to Lemma~\ref{lem:deps trunc}.  Since $m \in V \subseteq W$, it immediately follows that $m = \min \left( \trunc{f}{V}([0,\alpha)) \cap W \right)$, hence $\trunc{\trunc{f}{V}}{W}([0,\alpha)) = [m,f(\alpha))$ again by Lemma~\ref{lem:deps trunc}.  Thus, $\trunc{f}{V}([0,\alpha)) = \trunc{\trunc{f}{V}}{W}([0,\alpha))$.  The case where $f(\alpha) < 0$ is similar.  Thus, we have established condition (1) of Lemma~\ref{lem:same contour 2}.

Now let $\alpha' \in (0,1]$ be a right contour point of $\trunc{\trunc{f}{V}}{W}$.  By Lemmas~\ref{lem:deps trunc} and \ref{lem:contour trunc}, $\alpha'$ is a right departure of $\trunc{f}{V}$ and $\trunc{\trunc{f}{V}}{W}(\alpha') = \trunc{f}{V}(\alpha') = f(\alpha') \in W$.  We claim that $f(\alpha') \in V$.  Indeed, if not, then by Lemma~\ref{lem:deps trunc}, $\partial C_{\alpha'}^{f,V} = \{c_1,c_2\}$ for some $c_1,c_2 \in [0,1]$ with $c_1 < c_2$, and $c_2$ is a departure of $f$.  Now by Lemma~\ref{lem:deps trunc} (applied twice), $c_2$ is a departure of $\trunc{\trunc{f}{V}}{W}$ which has the same orientation as $\alpha'$.  In fact, every departure of $\trunc{\trunc{f}{V}}{W}$ in $(\alpha',c_2]$ has the same orientation as $\alpha'$.  But this contradicts the assumption that $\alpha'$ is a contour point of $\trunc{\trunc{f}{V}}{W}$.  Thus, $f(\alpha') \in V$.  As in the previous paragraph, it can now be seen that, in the case where $f(\alpha') > 0$, $\trunc{f}{V}([0,\alpha')) = \trunc{\trunc{f}{V}}{W}([0,\alpha')) = [m,f(\alpha'))$, where $m = \min \left( f([0,\alpha')) \cap V \right)$, and the case where $f(\alpha') < 0$ is similar.  This establishes condition (2) of Lemma~\ref{lem:same contour 2}.

By Lemma~\ref{lem:same contour 2}, this completes the proof that $\trunc{f}{V} {\restriction}_{[0,1]}$ and $\trunc{\trunc{f}{V}}{W} {\restriction}_{[0,1]}$ have the same contour factor.  The proof for the ``left halves'', $\trunc{f}{V} {\restriction}_{[-1,0]}$ and $\trunc{\trunc{f}{V}}{W} {\restriction}_{[-1,0]}$ (i.e.\ the proof that the contour factors of $\left( \trunc{f}{V} {\restriction}_{[-1,0]} \right) \circ \mathsf{r}$ and $\left( \trunc{\trunc{f}{V}}{W} {\restriction}_{[-1,0]} \right) \circ \mathsf{r}$ are the same), is similar.

For (4), let $x \in [-1,1]$.  We consider five cases.
\begin{itemize}
\item Case 1: $f(x) \in V$.  Here we immediately have $\trunc{f}{V}(x) = \trunc{f}{W}(x) = f(x)$, so $\trunc{\trunc{f}{V}}{W}(x) = \trunc{\trunc{f}{W}}{V}(x) = f(x)$.
\item Case 2: $f(x) \notin V$, $f(\partial C_x^{f,V}) = \{v\}$.  In this case $\trunc{f}{V}(x) = v \in V \subseteq W$, so $\trunc{\trunc{f}{V}}{W}(x) = v$.  If $\trunc{f}{W}(x) = v$, then we immediately have that $\trunc{\trunc{f}{W}}{V}(x) = v$.  Otherwise, by Lemma~\ref{lem:interval image trunc} $\trunc{f}{W}(C_x^{f,V}) \subseteq f(C_x^{f,V})$, and also $\trunc{f}{W}(\partial C_x^{f,V}) = f(\partial C_x^{f,V}) = \{v\}$.  It follows that $C_x^{\trunc{f}{W},V} = C_x^{f,V}$, and $\trunc{\trunc{f}{W}}{V}(x) = v$ also.
\item Case 3: $f(x) \notin V$, $f(\partial C_x^{f,V})$ has two elements.  In this case $\trunc{f}{V} = f$ on $\overline{C_x^{f,V}}$.  Note that in general, if $f = g$ on a closed interval $A \subseteq [-1,1]$ and $f(\partial A) \subseteq W$, then $\trunc{f}{W} = \trunc{g}{W}$ on $A$.  Here, this implies that $\trunc{\trunc{f}{V}}{W} = \trunc{f}{W}$ on $\overline{C_x^{f,V}}$ as well.  In particular, $\trunc{\trunc{f}{V}}{W}(x) = \trunc{f}{W}(x)$.  By Lemma~\ref{lem:interval image trunc} $\trunc{f}{W}(C_x^{f,V}) \subseteq f(C_x^{f,V})$, and also $\trunc{f}{W}(\partial C_x^{f,V}) = f(\partial C_x^{f,V})$.  It follows that $C_x^{\trunc{f}{W},V} = C_x^{f,V}$.  Hence, $\trunc{\trunc{f}{W}}{V}(x) = \trunc{f}{W}(x)$ also.
\end{itemize}

For \ref{fV tfWV}, note that $\trunc{f}{W}$ and $t_{\trunc{f}{W}}$ have the same radial contour factors (namely $t_{\trunc{f}{W}}$ itself), so by (2), the same is true for $\trunc{\trunc{f}{W}}{V}$ and $\trunc{t_{\trunc{f}{W}}}{V}$.  By (4), $\trunc{\trunc{f}{W}}{V} = \trunc{\trunc{f}{V}}{W}$, and this map has the same radial contour factor as $\trunc{f}{V}$ by (3).  Thus, $\trunc{f}{V}$ and $\trunc{t_{\trunc{f}{W}}}{V}$ have the same radial contour factors.
\end{proof}

\begin{lem}
\label{lem:refine trunc contours}
Let $f \colon [0,1] \to [-1,1]$ be a piecewise-linear map with $f(0) = 0$.  Let $V,W \subset [-1,1]$ be finite sets with $0 \in V \subseteq W$.  For each contour point $\alpha$ of $\trunc{f}{V}$, there exists a contour point $\gamma$ of $\trunc{f}{W}$ such that
\begin{enumerate}
\item $\alpha$ and $\gamma$ have the same orientation for $f$;
\item $\alpha \leq \gamma$; and
\item if $\alpha'$ is any contour point of $\trunc{f}{V}$ with $\alpha < \alpha'$, then $\gamma < \alpha'$.
\end{enumerate}

Consequently, if $\trunc{f}{V}$ and $\trunc{f}{W}$ have the same number of contour points, say the contour points of $\trunc{f}{V}$ are $0 < \alpha_1 < \cdots \alpha_n \leq 1$ and the contour points of $\trunc{f}{W}$ are $0 < \gamma_1 < \cdots \gamma_n \leq 1$, then for each $i$, $\alpha_i$ and $\gamma_i$ have the same orientation for $f$, and $|f(\alpha_i)| \leq |f(\gamma_i)|$.
\end{lem}

In particular, Lemma~\ref{lem:refine trunc contours} implies that $\trunc{f}{W}$ has at least as many contour points as $\trunc{f}{V}$.  See Figure~\ref{fig:refine trunc contours} for an illustration of a map truncated with respect to two finite sets $V \subset W$.

\begin{figure}
\begin{center}
\includegraphics{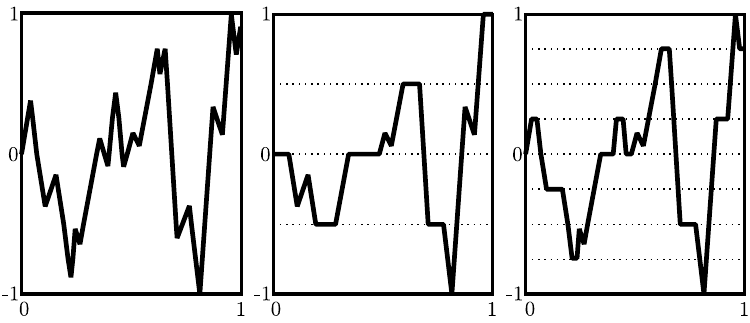}
\end{center}

\caption{On the left, a piecewise linear map $f \colon [0,1] \to [-1,1]$.  In the center, the truncation $\trunc{f}{V}$ of $f$ with respect to a finite set $V \subset [-1,1]$ whose elements are indicated by dotted horizontal lines (and also $-1,1 \in V$).  On the right, the truncation $\trunc{f}{W}$ of $f$ with respect to another finite set $W \supset V$.}
\label{fig:refine trunc contours}
\end{figure}

\begin{proof}
Let $\alpha$ be a contour point of $\trunc{f}{V}$.  By Lemma~\ref{lem:contour trunc}, $\alpha$ is a departure of $f$ and $f(\alpha) \in V$.  Since $V \subseteq W$, we have by Lemma~\ref{lem:deps trunc} that $\alpha$ is a departure of $\trunc{f}{W}$.  Therefore, there exists a contour point $\gamma$ of $\trunc{f}{W}$ with $\alpha \leq \gamma$, and such that for every departure $x$ of $\trunc{f}{W}$ with $\alpha \leq x \leq \gamma$, $x$ and $\alpha$ have the same orientation.  Now suppose there is a contour point $\alpha'$ of $\trunc{f}{V}$ with $\alpha < \alpha'$.  By taking $\alpha'$ to be the first such contour point, we have that $\alpha$ and $\alpha'$ have opposite orientations.  As above, we conclude that $\alpha'$ is a departure of $\trunc{f}{W}$.  Therefore, by the properties defining $\gamma$, we must have $\gamma < \alpha'$.
\end{proof}

\section{Mesh, truncation, and contour factors}
\label{sec:mesh trunc}

In this section we collect some further results about truncations, concentrating on those which involve taking a preimage of a finite set $V \subset [-1,1]$.

For $x,y \in [-1,1]$ and $\varepsilon > 0$, we say $x$ and $y$ are \emph{$\varepsilon$-close} if $|x-y| \leq \varepsilon$.  We say two maps $f,g \colon [a,b] \to [-1,1]$ are $\varepsilon$-close if $f(x)$ and $g(x)$ are $\varepsilon$-close for all $x \in [a,b]$.

In the following Definitions and Lemma, let $[a,b]$ represent either $[0,1]$ or $[-1,1]$.

\begin{defn}
If $V \subset [a,b]$ is a non-empty finite set, define the \emph{mesh} of $V$ by
\[ \mesh(V) = \max\{v_2 - v_1: v_1,v_2 \in V \cup \{a,b\}, \; v_1 < v_2, \; (v_1,v_2) \cap V = \emptyset\} .\]
\end{defn}

Equivalently, we have that $\mesh(V) \leq L$ if and only if for any closed interval $A \subseteq [a,b]$ of length $L$, $V \cap A \neq \emptyset$.

\begin{defn}
Let $V \subset [a,b]$ be a finite set.  For $x,y \in [a,b]$, we say that $x$ and $y$ are \emph{$V$-close}, written $x \overset{V}{=} y$, if there is no point of $V$ strictly between $x$ and $y$.  That is, either
\begin{itemize}
\item $x = y$; or
\item $x < y$ and $(x,y) \cap V = \emptyset$; or
\item $y < x$ and $(y,x) \cap V = \emptyset$.
\end{itemize}
If $V \subset [-1,1]$ and $f,g \colon [a,b] \to [-1,1]$, we say $f,g$ are \emph{$V$-close}, written $f \overset{V}{=} g$, if $f(x) \overset{V}{=} g(x)$ for all $x \in [a,b]$.
\end{defn}

Clearly, if $x$ and $y$ are $V$-close, then they are $\mesh(V)$-close.

\begin{lem}
\label{lem:V-close props}
Let $f \colon [a,b] \to [-1,1]$ be a piecewise-linear map, and let $V,W \subset [-1,1]$ be finite sets with $V \subseteq W$ and such that $f([a,b]) \cap V \neq \emptyset$.  Let $x,y,z \in [-1,1]$.  Then:
\begin{enumerate}[label=(\arabic{*})]
\item if $x \overset{V}{=} y$, $y \overset{V}{=} z$, and $y \notin V$, then $x \overset{V}{=} z$;
\item if $x \overset{W}{=} y$, then $x \overset{V}{=} y$;
\item $f \overset{V}{=} \trunc{f}{V}$;
\item $\trunc{f}{W} \overset{V}{=} \trunc{f}{V}$;
\item \label{f(x) f(y) V-close} if $f^{-1}(V)$ is finite (e.g.\ if $f$ is nowhere constant), and if $x \overset{f^{-1}(V)}{=\joinrel=} y$, then $f(x) \overset{V}{=} f(y)$; and
\item \label{f(x) truncf(y) V-close} if $f^{-1}(V)$ is finite (e.g.\ if $f$ is nowhere constant), and if $x \overset{f^{-1}(V)}{=\joinrel=} y$, then $f(x) \overset{V}{=} \trunc{f}{V}(y)$.
\end{enumerate}
\end{lem}

\begin{proof}
Claims (1) and (2) are straightforward and left to the reader.

For (3), let $x \in [a,b]$.  Suppose $\trunc{f}{V}(x) \neq f(x)$.  This means that $f(x) \notin V$, $f(\partial C_x) = \{v\}$ for some $v \in V$, and $\trunc{f}{V}(x) = v$.  The open interval with endpoints $f(x)$ and $v$ is contained in $f(C_x)$, which is disjoint from $V$ by definition.  Thus $f(x) \overset{V}{=} \trunc{f}{V}(x)$.

For (4), let $x \in [a,b]$.  If $f(x) \in V$, then $\trunc{f}{V}(x) = \trunc{f}{W}(x) = f(x)$.  Suppose then that $f(x) \notin V$.  By (3), $f(x) \overset{W}{=} \trunc{f}{W}(x)$, so $f(x) \overset{V}{=} \trunc{f}{W}(x)$ by (2).  Also by (3), $f(x) \overset{V}{=} \trunc{f}{V}(x)$.  Therefore by (1), $\trunc{f}{W}(x) \overset{V}{=} \trunc{f}{V}(x)$.

For (5), suppose $f^{-1}(V)$ is finite and $x \overset{f^{-1}(V)}{=\joinrel=} y$, and that $f(x) \neq f(y)$ (in particular $x \neq y$).  Let $U$ be the open interval with endpoints $x$ and $y$.  Then $U \cap f^{-1}(V) = \emptyset$, so $f(U) \cap V = \emptyset$.  Since the open interval with endpoints $f(x)$ and $f(y)$ is contained in $f(U)$, we conclude that $f(x) \overset{V}{=} f(y)$.

For (6), suppose $f^{-1}(V)$ is finite and $x \overset{f^{-1}(V)}{=\joinrel=} y$.  By (5), $f(x) \overset{V}{=} f(y)$.  If $f(y) \in V$, then $\trunc{f}{V}(y) = f(y)$, and we are done.  Suppose then that $f(y) \notin V$.  By (3), $f(y) \overset{V}{=} \trunc{f}{V}(y)$.  Therefore by (1), $f(x) \overset{V}{=} \trunc{f}{V}(y)$.
\end{proof}

The next Lemma will be applied in the setting where $t = t_{\trunc{f}{V}}$ is the radial contour factor of $\trunc{f}{V}$, but since the proof does not involve this specific factorization of $\trunc{f}{V}$, we state the result in a slightly more general form for simplicity.

\begin{lem}
\label{lem:f s t V-close}
Let $V \subset [-1,1]$ be a finite set with $0 \in V$, and let $f,t,s \colon [-1,1] \to [-1,1]$ be piecewise-linear maps such that $f(0) = t(0) = s(0) = 0$, $f$ and $t$ are nowhere constant, and $\trunc{f}{V} = t \circ s$.  If $x,y \in [-1,1]$ and $x \overset{f^{-1}(V)}{=\joinrel=} y$, then $s(x) \overset{t^{-1}(V)}{=\joinrel=} s(y)$.
\end{lem}

\begin{proof}
Let $x,y \in [-1,1]$, and suppose $x \overset{f^{-1}(V)}{=\joinrel=} y$.  Assume without loss of generality that $x < y$.  Now $f((x,y)) \cap V = \emptyset$, and so either $\trunc{f}{V} = f$ on $(x,y)$, or $\trunc{f}{V}$ is constant on $(x,y)$.  If $\trunc{f}{V} = f$ on $(x,y)$, then $f = t \circ s$ on $(x,y)$, and it follows easily that $s((x,y))$, which contains the interval between $s(x)$ and $s(y)$ (if $s(x) \neq s(y)$), is disjoint from $t^{-1}(V)$, hence $s(x) \overset{t^{-1}(V)}{=\joinrel=} s(y)$.  If $\trunc{f}{V}$ is constant on $(x,y)$, then since $\trunc{f}{V} = t \circ s$ and $t$ is nowhere constant, it follows that $s$ is constant on $(x,y)$, and so $s(x) = s(y)$.
\end{proof}

\begin{lem}
\label{lem:mesh preim contour}
Let $f \colon [0,1] \to [-1,1]$ be a non-constant piecewise-linear map with $f(0) = 0$.  Let $V \subset [-1,1]$ be a finite set with $0 \in V$ and such that $\trunc{f}{V}$ is non-constant, and let $0 < \alpha_1 < \cdots < \alpha_n \leq 1$ be the contour points of $\trunc{f}{V}$.  Then
\[ \mesh \left( \left( t_{\trunc{f}{V}} \right)^{-1}(V) \right) \leq \frac{\mesh(V)}{n \cdot |f(\alpha_1)|} .\]
\end{lem}

\begin{proof}
Recall from Lemma~\ref{lem:contour trunc} that for each $i =1,\ldots,n$, $\trunc{f}{V}(\alpha_i) = f(\alpha_i) \in V$.

The map $t_{\trunc{f}{V}}$ is given as follows: $t_{\trunc{f}{V}}(0) = 0$, $t_{\trunc{f}{V}}(\frac{i}{n}) = f(\alpha_i)$ for each $i = 1,\ldots,n$, and $t_{\trunc{f}{V}}$ is linear in between these points.  Note that the values $f(\alpha_i)$, for odd $i \in \{1,\ldots,n\}$, have the same sign, and $0 < |f(\alpha_1)| < |f(\alpha_3)| < \cdots$.  Similarly, the values $f(\alpha_i)$, for even $i \in \{1,\ldots,n\}$ have the same sign, which is opposite that of the values of the odd-numbered contour points, and $0 < |f(\alpha_2)| < |f(\alpha_4)| < \cdots$.  Consequently, the magnitude of the slope of each linear section of $t_{\trunc{f}{V}}$ is greater than or equal to $\frac{|f(\alpha_1)|}{1/n} = n \cdot |f(\alpha_1)|$, i.e.\ the magnitude of the slope of the first linear section.

Let $A \subseteq [0,1]$ be any closed interval of length $L = \frac{\mesh(V)}{n \cdot |f(\alpha_1)|}$.  If $\frac{i}{n} \in A$ for some $i \in \{1,\ldots,n\}$, then since $t_{\trunc{f}{V}}(\frac{i}{n}) = f(\alpha_i) \in V$, we have $\left( t_{\trunc{f}{V}} \right)^{-1}(V) \cap A \neq \emptyset$.  Otherwise, $A$ is contained in one linear section of $t_{\trunc{f}{V}}$, and by the slope estimate above, the length of the interval $t_{\trunc{f}{V}}(A)$ is greater than or equal to $L \cdot n \cdot |f(\alpha_1)| = \mesh(V)$.  Therefore, $t_{\trunc{f}{V}}(A) \cap V \neq \emptyset$, hence in this case also, $\left( t_{\trunc{f}{V}} \right)^{-1}(V) \cap A \neq \emptyset$.  We conclude that $\mesh \left( \left( t_{\trunc{f}{V}} \right)^{-1}(V) \right) \leq \frac{\mesh(V)}{n \cdot |f(\alpha_1)|}$, as desired.
\end{proof}

\section{Approximation of bonding maps}
\label{sec:approx bonding maps}

In this section, we discuss a method for constructing alternative inverse limit representations of arc-like continua via suitable approximations of their bonding maps.

Let $f_k \colon [-1,1] \to [-1,1]$, $k \in \mathbb{N}$, be piecewise-linear maps with $f_k(0) = 0$ for all $k$.  For $k < \ell$ we denote $f_k^\ell = f_k \circ f_{k+1} \circ \cdots \circ f_{\ell-1}$.  In this notation, $f_k^k = \mathrm{id}$, $f_k = f_k^{k+1}$, and for $j \leq k \leq \ell$, $f_j^\ell = f_j^k \circ f_k^\ell$.

\begin{thm}[Special case of Mioduszewski's Theorem, \cite{mioduszewski1963}]
\label{thm:mioduszewski}
Let $X = \varprojlim \left \langle [-1,1],f_k \right \rangle$ and $Y = \varprojlim \left \langle [-1,1],g_k \right \rangle$.  If there exists a sequence $\langle \varepsilon_k \rangle_{k \geq 1}$ of positive numbers such that $\varepsilon_k \to 0$ as $k \to\infty$, and such that for every $j \leq k \leq \ell$,
\begin{enumerate}[label=(\arabic{*})]
\item \label{mio1} $f_j^\ell$ and $f_j^k \circ g_k^\ell$ are $\varepsilon_k$-close, and
\item \label{mio2} $g_j^\ell$ and $g_j^k \circ f_k^\ell$ are $\varepsilon_k$-close,
\end{enumerate}
then $X$ and $Y$ are homeomorphic.  In fact, a homeomorphism between $X$ and $Y$ is given by $\langle x_k \rangle_{k \geq 1} \mapsto \left \langle \lim_{\ell \to \infty} g_k^\ell(x_\ell) \right \rangle_{k \geq 1}$.  Consequently, if $f_k(0) = 0 = g_k(0)$ for all $k \geq 1$, then the homeomorphism sends $\langle 0,0,\ldots \rangle \in X$ to $\langle 0,0,\ldots \rangle \in Y$.
\end{thm}

We visually depict the scenario in Theorem~\ref{thm:mioduszewski} using an ``almost commutative'' diagram as in Figure~\ref{fig:mioduszewski diagram a}.  The presence of the $\varepsilon_k$'s in this diagram is meant as an abbreviation, suggesting that all of the diagrams of the forms represented in Figures~\ref{fig:mioduszewski diagram b} and \ref{fig:mioduszewski diagram c} are $\varepsilon_k$-commutative.

\begin{figure}
\begin{center}
\begin{subfigure}{\linewidth}
\begin{center}
\begin{tikzpicture}[>=stealth', scale=1.5]
\foreach \x [count=\xi] in {0,1.5,...,6} {
  \node (row1-\xi) at (\x,1) {$I$};
  \node (row2-\xi) at (\x,0) {$I$};
}
\node (row1-6) at (7.5,1) {$\cdots$};
\node (row2-6) at (7.5,0) {$\cdots$};

\foreach \x [remember=\x as \lastx (initially 1)] in {2,...,6} {
  \draw[->] (row1-\x) to node[above]{$f_{\lastx}$} (row1-\lastx);
  \draw[->] (row2-\x) to node[below]{$g_{\lastx}$} (row2-\lastx);
}

\foreach \x in {1,...,5}
  \draw[<->, dashed] (row1-\x)--(row2-\x);

\foreach \x [count=\xi] in {0.75,2.25,...,6.75}
  \node at (\x,0.5) {\small $\varepsilon_\xi$};
\end{tikzpicture}
\caption{}
\label{fig:mioduszewski diagram a}
\end{center}
\end{subfigure}\\

\begin{subfigure}{\linewidth}
\begin{center}
\begin{tikzpicture}[>=stealth', scale=1.5]
\node (row1-1) at (0,1) {$I$};
\node (row1-2) at (1,1) {$I$};
\node (row1-3) at (1.75,1) {$\cdots$};
\node (row1-4) at (2.5,1) {$I$};
\node (row1-5) at (3.5,1) {$I$};
\node (row1-6) at (4.5,1) {$I$};
\node (row1-7) at (5.25,1) {$\cdots$};
\node (row1-8) at (6,1) {$I$};
\node (row1-9) at (7,1) {$I$};

\draw[->] (row1-2) to node[above]{$f_j$} (row1-1);
\draw[->] (row1-3) to (row1-2);
\draw[->] (row1-4) to (row1-3);
\draw[->] (row1-5) to node[above]{$f_{k-1}$} (row1-4);
\draw[->] (row1-6) to node[above]{$f_k$} (row1-5);
\draw[->] (row1-7) to (row1-6);
\draw[->] (row1-8) to (row1-7);
\draw[->] (row1-9) to node[above]{$f_{\ell-1}$} (row1-8);

\node (row2-5) at (3.5,0) {$I$};
\node (row2-6) at (4.5,0) {$I$};
\node (row2-7) at (5.25,0) {$\cdots$};
\node (row2-8) at (6,0) {$I$};
\node (row2-9) at (7,0) {$I$};

\draw[->] (row2-6) to node[below]{$g_k$} (row2-5);
\draw[->] (row2-7) to (row2-6);
\draw[->] (row2-8) to (row2-7);
\draw[->] (row2-9) to node[below]{$g_{\ell-1}$} (row2-8);

\draw[<->, dashed] (row1-5)--(row2-5);
\draw[<->, dashed] (row1-9)--(row2-9);
\end{tikzpicture}
\caption{}
\label{fig:mioduszewski diagram b}
\end{center}
\end{subfigure}\\

\begin{subfigure}{\linewidth}
\begin{center}
\begin{tikzpicture}[>=stealth', scale=1.5]
\node (row1-5) at (3.5,1) {$I$};
\node (row1-6) at (4.5,1) {$I$};
\node (row1-7) at (5.25,1) {$\cdots$};
\node (row1-8) at (6,1) {$I$};
\node (row1-9) at (7,1) {$I$};

\draw[->] (row1-6) to node[above]{$f_k$} (row1-5);
\draw[->] (row1-7) to (row1-6);
\draw[->] (row1-8) to (row1-7);
\draw[->] (row1-9) to node[above]{$f_{\ell-1}$} (row1-8);

\node (row2-1) at (0,0) {$I$};
\node (row2-2) at (1,0) {$I$};
\node (row2-3) at (1.75,0) {$\cdots$};
\node (row2-4) at (2.5,0) {$I$};
\node (row2-5) at (3.5,0) {$I$};
\node (row2-6) at (4.5,0) {$I$};
\node (row2-7) at (5.25,0) {$\cdots$};
\node (row2-8) at (6,0) {$I$};
\node (row2-9) at (7,0) {$I$};

\draw[->] (row2-2) to node[below]{$g_j$} (row2-1);
\draw[->] (row2-3) to (row2-2);
\draw[->] (row2-4) to (row2-3);
\draw[->] (row2-5) to node[below]{$g_{k-1}$} (row2-4);
\draw[->] (row2-6) to node[below]{$g_k$} (row2-5);
\draw[->] (row2-7) to (row2-6);
\draw[->] (row2-8) to (row2-7);
\draw[->] (row2-9) to node[below]{$g_{\ell-1}$} (row2-8);

\draw[<->, dashed] (row1-5)--(row2-5);
\draw[<->, dashed] (row1-9)--(row2-9);
\end{tikzpicture}
\caption{}
\label{fig:mioduszewski diagram c}
\end{center}
\end{subfigure}
\end{center}

\caption{First picture: diagram depicting, in an abbreviated way, the conditions of Theorem~\ref{thm:mioduszewski}.  Here, $I$ denotes the interval $[-1,1]$, and dashed arrows are identities.  Second and third pictures: the explicit $\varepsilon_k$-commutative diagrams corresponding to conditions \ref{mio1} and \ref{mio2} of Theorem~\ref{thm:mioduszewski}.}
\label{fig:mioduszewski diagram}
\end{figure}
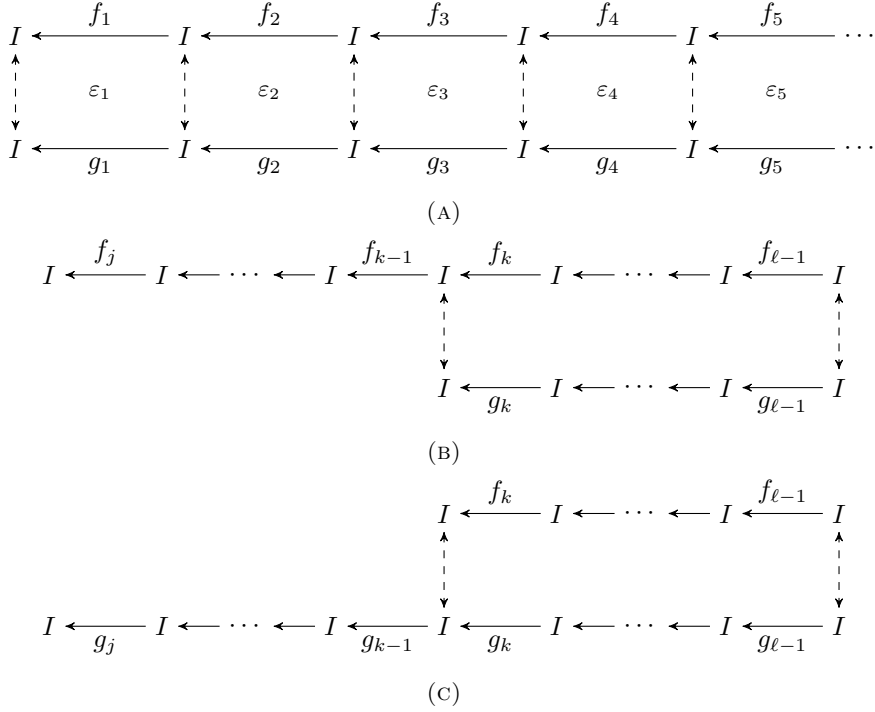

\begin{proof}
In the notation established here, the theorem of Mioduszewski \cite[Theorem~4]{mioduszewski1963} states:

\medskip
Let $\langle X_m,f_m \rangle$ and $\langle Y_n,g_n \rangle$ be inverse systems of polyhedra, let $X = \varprojlim \left \langle X_m,f_m \right \rangle$, $Y = \varprojlim \left \langle Y_n,g_n \right \rangle$, and let $\varepsilon_k > 0$ be numbers such that $\varepsilon_k \to 0$.  Suppose there exist sequences $\langle m_k \rangle_{k=1}^\infty$ and $\langle n_k \rangle_{k=1}^\infty$ and maps $\sigma_k \colon X_{m_k} \to Y_{n_k}$ for each odd $k$ and maps $\tau_k \colon Y_{n_k} \to X_{m_k}$ for each even $k$, such that for every $j \leq k \leq \ell$,
\begin{enumerate}[label=(M\arabic{*})]
\item \label{orig mio1} if $k$ is odd and $\ell$ is even, then $g_{n_j}^{n_\ell}$ and $g_{n_j}^{n_k} \circ \sigma_k \circ f_{m_k}^{m_\ell} \circ \tau_\ell$ and are $\varepsilon_k$-close,
\item \label{orig mio2} if $k$ is even and $\ell$ is odd, then $f_{m_j}^{m_\ell}$ and $f_{m_j}^{m_k} \circ \tau_k \circ g_{n_k}^{n_\ell} \circ \sigma_\ell$ and are $\varepsilon_k$-close,
\item \label{orig mio3} if both $k$ and $\ell$ are even, then $f_{m_j}^{m_\ell} \circ \tau_\ell$ and $f_{m_j}^{m_k} \circ \tau_k \circ g_{n_k}^{n_\ell}$ are $\varepsilon_k$-close, and
\item \label{orig mio4} if both $k$ and $\ell$ are odd, then $g_{n_j}^{n_\ell} \circ \sigma_\ell$ and $g_{n_j}^{n_k} \circ \sigma_k \circ f_{m_k}^{m_\ell}$ are $\varepsilon_k$-close.
\end{enumerate}
Then $X \approx Y$, and, in fact, there exists a homeomorphism $\sigma \colon X \to Y$ such that for every $k \leq \ell$,
\begin{itemize}
\item if $\ell$ is odd, then $g_{n_k}^{n_\ell} \circ \sigma_\ell \circ \pi_{X,m_\ell}$ and $\pi_{Y,n_k} \circ \sigma$ are $\varepsilon_\ell$-close, and
\item if $\ell$ is even, then $f_{m_k}^{m_\ell} \circ \tau_\ell \circ \pi_{Y,n_\ell}$ and $\pi_{X,m_k} \circ \sigma^{-1}$ are $\varepsilon_\ell$-close,
\end{itemize}
where $\pi_{X,m} \colon X \to X_m$ and $\pi_{Y,n} \colon Y \to Y_n$ denote the $m$-th and $n$-th coordinate projections, respectively.
\medskip

We apply this Theorem in the present setting, where $X_m = Y_n = [-1,1]$ for each $m$ and $n$, $m_k = n_k = k$ for each $k$, and $\sigma_k = \tau_k = \mathrm{id}$ for each $k$.  Clearly the conditions \ref{mio1} and \ref{mio2} of the present Theorem imply conditions \ref{orig mio1}--\ref{orig mio4} above.  We conclude that $X \approx Y$, and moreover there exists a homeomorphism $\sigma \colon X \to Y$ such that for every $k$ and every odd $\ell \geq k$, $g_k^\ell \circ \pi_{X,\ell}$ and $\pi_{Y,k} \circ \sigma$ are $\varepsilon_\ell$-close.  In other words, for any $\langle x_i \rangle_{i \geq 1} \in X$, for any odd $\ell \geq k$, the distance between the $k$-th coordinate of $\sigma \left( \langle x_i \rangle_{i \geq 1} \right)$ and $g_k^\ell(x_\ell)$ is less than $\varepsilon_\ell$.  Also, given $k \leq \ell \leq \ell'$, the distance between $g_k^\ell(x_\ell) = g_k^\ell \circ f_\ell^{\ell'}(x_{\ell'})$ and $g_k^{\ell'}(x_{\ell'})$ is less than $\varepsilon_\ell$ by condition \ref{mio2}.  Thus, since $\varepsilon_\ell \to 0$, $\langle g_k^\ell(x_\ell) \rangle_{\ell \geq k}$ is a Cauchy sequence, and so $\lim_{\ell \to \infty} g_k^\ell(x_\ell)$ exists.  We deduce that $\sigma \left( \langle x_k \rangle_{k \geq 1} \right) = \left \langle \lim_{\ell \to \infty} g_k^\ell(x_\ell) \right \rangle_{k \geq 1}$.
\end{proof}

The next Lemma is one application of Theorem~\ref{thm:mioduszewski}.  It will be used to dispense with a simple case at the beginning of the proof of Theorem~\ref{thm:main} in the next section, so that the rest of that proof can focus on the non-trivial case where the theory of radial departures and radial contour factors applies.  We will also use Theorem~\ref{thm:mioduszewski} later in the proof of Theorem~\ref{thm:main}.

\begin{lem}
\label{lem:nearly constant half}
Let $f_i \colon [-1,1] \to [-1,1]$, $i \in \mathbb{N}$, be maps with $f_i(0) = 0$ for all $i \in \mathbb{N}$.  Suppose that for infinitely many $i \in \mathbb{N}$, the following property holds:
\begin{enumerate}[label=($\dagger$)]
\item \label{converge constant half} for each $\varepsilon > 0$, there exists $n > i$ such that $|f_i^n(x)| < \varepsilon$ for all $x \in [0,1]$.
\end{enumerate}
Then there exist maps $g_k \colon [-1,1] \to [-1,1]$, $k \in \mathbb{N}$, such that $g_k {\restriction}_{[0,1]}$ is $0$ (constant map) for each $k \in \mathbb{N}$, and such that the inverse limits $X = \varprojlim \left \langle [-1,1],f_i \right \rangle$ and $Y = \varprojlim \left \langle [-1,1],g_k \right \rangle$ are homeomorphic.  In fact, there exists a homeomorphism which maps the point $\langle 0,0,\ldots \rangle \in X$ to $\langle 0,0,\ldots \rangle \in Y$.
\end{lem}

\begin{proof}
Fix any sequence $\langle \varepsilon_k \rangle_{k \geq 1}$ of positive numbers such that $\varepsilon_k \to 0$ as $k \to\infty$ (e.g.\ let $\varepsilon_k = \frac{1}{k}$).

We will construct by recursion an increasing sequence $\langle i_k \rangle_{k=1}^\infty$ of positive integers, and maps $g_k \colon [-1,1] \to [-1,1]$, $k = 1,2,\ldots$, with properties described below.  For each $k$, denote by $F_k$ the composition $f_{i_k}^{i_{k+1}}$.  The properties we will require are that $g_k {\restriction}_{[0,1]}$ is $0$ (constant map) for each $k$, and for every $j \leq k \leq \ell$,
\begin{enumerate}
\item $F_j^\ell$ and $F_j^k \circ g_k^\ell$ are $\varepsilon_k$-close; and
\item $g_j^\ell$ and $g_j^k \circ F_k^\ell$ are $\varepsilon_k$-close.
\end{enumerate}

To begin, let $i_1 = 1$.

Let $\ell > 1$, and suppose that $i_k$ has been defined for all $1 \leq k < \ell$, and $g_k$ has been defined for all $1 \leq k < \ell-1$.  Let $i \geq i_{\ell-1}$ be such that \ref{converge constant half} holds, and choose $i_\ell > i$ such that for all $x \in [0,1]$, $|f_i^{i_\ell}(x)|$ is so small that
\[ |F_j^\ell(x)| < \varepsilon_k \textrm{ for each } 1 \leq j \leq k < \ell .\]
Recall that $F_j^\ell = F_j \circ \cdots \circ F_{\ell-1} = f_{i_j}^{i_{j+1}} \circ \cdots \circ f_{i_{\ell-1}}^{i_\ell} = f_{i_j}^{i_\ell} = f_{i_j}^i \circ f_i^{i_\ell}$.  Define $g_{\ell-1} \colon [-1,1] \to [-1,1]$ by
\[ g_{\ell-1}(x) = \begin{cases}
F_{\ell-1}^\ell(x) & \textrm{if } x \in [-1,0] \\
0 & \textrm{if } x \in [0,1] .
\end{cases} \]

Let $x \in [-1,1]$.  Given $1 \leq k_0 < \ell$, if $F_k^\ell(x) \in [-1,0]$ for each $k_0 < k \leq \ell$, then $g_k^\ell(x) = F_k^\ell(x)$ for each $k_0 \leq k \leq \ell$.  In particular, if $F_k^\ell(x) \in [-1,0]$ for each $1 < k \leq \ell$, then for any $j \leq k \leq \ell$, $F_j^\ell(x) = F_j^k \circ g_k^\ell(x)$ and $g_j^\ell(x) = g_j^k \circ F_k^\ell(x)$.

Suppose now that $F_{k_0}^\ell(x) \in [0,1]$ for some $1 < k_0 \leq \ell$, and let $k_0$ be the maximal such integer.  Let $j \leq k \leq \ell$.  If $k_0 \leq k$, then again $g_k^\ell(x) = F_k^\ell(x)$, and so $F_j^\ell(x) = F_j^k \circ g_k^\ell(x)$ and $g_j^\ell(x) = g_j^k \circ F_k^\ell(x)$ as above.  If $k < k_0$, then
\[ g_k^\ell(x) = g_k^{k_0-1} \circ g_{k_0-1} \circ g_{k_0}^\ell(x) = g_k^{k_0-1} \circ g_{k_0-1} \circ F_{k_0}^\ell(x) = g_k^{k_0-1}(0) = 0 \]
and so also $g_j^\ell(x) = g_j^k \circ g_k^\ell(x) = g_j^k(0) = 0$ and $F_j^k \circ g_k^\ell(x) = F_j^k(0) = 0$.  By construction, $|F_j^\ell(x)| = |F_j^{k_0} \circ F_{k_0}^\ell(x)| < \varepsilon_k$, since $F_{k_0}^\ell(x) \in [0,1]$.  Also, $g_j^k \circ F_k^\ell(x)$ is either $0$ or equal to $F_j^\ell(x)$, hence $|g_j^k \circ F_k^\ell(x)| < \varepsilon_k$ as well.  We have now established (1) and (2) in all cases.  This completes the recursive construction.

We point out that $X = \varprojlim \left \langle [-1,1],f_i \right \rangle$ is homeomorphic to $\varprojlim \left \langle [-1,1],F_k \right \rangle$, as the maps $F_k$ are compositions of the bonding maps $f_i$.  In light of the properties (1) and (2) above, the conclusion of the Lemma now follows from Theorem~\ref{thm:mioduszewski}.
\end{proof}

\section{Proof of the Main Theorem}
\label{sec:main proof}

In this section we prove the main theorem of this paper, Theorem~\ref{thm:main}.  We begin with some technical lemmas.

Recall that, given a map $f \colon [-1,1] \to [-1,1]$ with $f(0) = 0$ and a finite set $V \subset [-1,1]$ with $0 \in V$, $\trunc{f}{V}$ has a well-defined radial contour factor if the restrictions $\trunc{f}{V} {\restriction}_{[0,1]}$ and $\trunc{f}{V} {\restriction}_{[-1,0]}$ are non-constant, or equivalently, if both of the sets $f([0,1]) \cap V \smallsetminus \{0\}$ and $f([-1,0]) \cap V \smallsetminus \{0\}$ are non-empty.

\begin{lem}
\label{lem:well-def preimage}
Let $f,g \colon [-1,1] \to [-1,1]$ be piecewise-linear maps with $f(0) = g(0) = 0$.  Let $V \subset [-1,1]$ be a finite set with $0 \in V$, and suppose $V' := f^{-1}(V)$ is finite (e.g.\ if $f$ is nowhere constant).  If $\trunc{f \circ g}{V}$ has well-defined radial contour factor, then $\trunc{g}{V'}$ has well-defined radial contour factor.
\end{lem}

\begin{proof}
Suppose that $\trunc{f \circ g}{V}$ has well-defined radial contour factor, meaning that the sets $f \circ g([0,1]) \cap V \smallsetminus \{0\}$ and $f \circ g([-1,0]) \cap V \smallsetminus \{0\}$ are non-empty.  It follows immediately that the sets $g([0,1]) \cap V' \smallsetminus \{0\}$ and $g([-1,0]) \cap V' \smallsetminus \{0\}$ are non-empty.  Thus, $\trunc{g}{V'}$ has well-defined radial contour factor.
\end{proof}

Lemma~\ref{lem:V J} below is technically stronger than the following Lemma~\ref{lem:same trunc contour}, because it can be argued from Lemma~\ref{lem:compare trunc} that if $V \subseteq W$ and $\trunc{f}{W}$ and $\trunc{g}{W}$ have the same radial contour factor, then $\trunc{f}{V}$ and $\trunc{g}{V}$ have the same radial contour factor as well.  However, we include Lemma~\ref{lem:same trunc contour} anyway for its simplicity, and we use it in the proof of Lemma~\ref{lem:V J}, and again in the proof of Theorem~\ref{thm:main} below.

\begin{lem}
\label{lem:same trunc contour}
Let $\mathcal{F}$ be an infinite family of piecewise-linear maps from $[-1,1]$ to $[-1,1]$ with $f(0) = 0$ for each $f \in \mathcal{F}$, and let $V \subset [-1,1]$ be a finite set with $0 \in V$.  Suppose that $\trunc{f}{V}$ has well-defined radial contour factor for all $f \in \mathcal{F}$.  There exists an infinite set $\mathcal{F}' \subseteq \mathcal{F}$ such that all of the maps $\trunc{f}{V}$, $f \in \mathcal{F}'$, have the same radial contour factor.
\end{lem}

\begin{proof}
According to Lemma~\ref{lem:contour trunc}, all of the radial contour points of the maps $\trunc{f}{V}$, $f \in \mathcal{F}$, have images in the set $V$.  Since $V$ is finite, this means that there are only finitely many possibilities for the radial contour factors of these maps.  The Lemma follows immediately by the pigeonhole principle.
\end{proof}

\begin{lem}
\label{lem:V J}
Let $\mathcal{F}$ be an infinite family of piecewise-linear maps from $[-1,1]$ to $[-1,1]$ with $f(0) = 0$ for each $f \in \mathcal{F}$, let $V \subset [-1,1]$ be a finite set with $0 \in V$, and let $\delta > 0$.  Suppose that $\trunc{f}{V}$ has well-defined radial contour factor for all $f \in \mathcal{F}$.  Then there exists a finite set $W \subset [-1,1]$ with $V \subseteq W$, and an infinite set $\mathcal{F}' \subseteq \mathcal{F}$, such that:
\begin{enumerate}
\item All of the maps $\trunc{f}{W}$, $f \in \mathcal{F}'$, have the same radial contour factor $t$; and
\item $\mesh \left( t^{-1}(W) \right) < \delta$.
\end{enumerate}
\end{lem}

We remark that in Lemma~\ref{lem:V J}, since $V$ is arbitrary, $W$ can be chosen so that $\mesh(W)$ is arbitrarily small.

\begin{proof}
Let $V_1,V_2,\ldots$ be a sequence of finite subsets of $[-1,1]$ such that $V \subseteq V_1 \subset V_2 \subset \cdots$, and $\mesh(V_k) \to 0$ as $k \to \infty$.  Let $\mathcal{F}_0 = \mathcal{F}$.  Recursively, for each $k \geq 1$, apply Lemma~\ref{lem:same trunc contour} to obtain an infinite subset $\mathcal{F}_k \subseteq \mathcal{F}_{k-1}$ such that all of the maps $\trunc{f}{V_k}$, $f \in \mathcal{F}_k$, have the same radial contour factor $t_k$.

We will prove that for sufficiently large $k$, $\mesh \left( t_k^{-1}(V_k) \right) < \delta$, thus establishing properties (1) and (2) of the Lemma (with $V_k$, $\mathcal{F}_k$, and $t_k$ playing the roles of $W$, $\mathcal{F}'$, and $t$ there).  We will consider the ``right side'', $t_k {\restriction}_{[0,1]}$, of the maps $t_k$, only.  That is, we will prove that for sufficiently large $k$, $\mesh \left( t_k^{-1}(V_k) \cap [0,1] \right) < \delta$.  The argument for the ``left side'' is similar.  We consider two cases.

\medskip
\textbf{Case~1:} Suppose the number of right contour points of $t_k$ is unbounded as $k \to \infty$.  Note that the right contour points of $t_k$ are evenly spaced in $[0,1]$ (by definition of contour factor).  Also, for any $f \in \mathcal{F}_k$, the images of these right contour points under $t_k$ are equal to the images, under $\trunc{f}{V_k}$, of the right contour points of $\trunc{f}{V_k}$, and these belong to $V_k$ by Lemmas~\ref{lem:deps trunc} and \ref{lem:contour trunc}.  It follows that for sufficiently large $k$, $\mesh \left( t_k^{-1}(V_k)  \cap [0,1] \right) < \delta$.

\medskip
\textbf{Case~2:} Suppose the number of right contour points of $t_k$ is bounded as $k \to \infty$.  Let $k_0 \in \mathbb{N}$ be such that the number $N$ of right contour points of $t_{k_0}$ is maximal among all $t_k$, $k \in \mathbb{N}$.  We claim that for all $k > k_0$, $t_k$ also has $N$ right contour points.  Indeed, given $k > k_0$, let $f \in \mathcal{F}_k$, so that $t_k$ is the radial contour factor of $\trunc{f}{V_k}$.  Since $\mathcal{F}_k \subset \mathcal{F}_{k_0}$, $f \in \mathcal{F}_{k_0}$ as well, so $t_{k_0}$ is the radial contour factor of $\trunc{f}{V_{k_0}}$.  Hence, the number of right contour points of $\trunc{f}{V_{k_0}}$ is $N$.  According to the remark immediately following the statement of Lemma~\ref{lem:refine trunc contours}, the number of right contour points of $\trunc{f}{V_k}$ (and thus of $t_k$) is at least $N$, hence equals $N$ by maximality of $N$.

Now let $k_1,k_2 \in \mathbb{N}$ with $k_0 \leq k_1 < k_2$.  Let $f \in \mathcal{F}_{k_2}$, so that $t_{k_1}$ is the radial contour factor of $\trunc{f}{V_{k_1}}$, and $t_{k_2}$ is the radial contour factor of $\trunc{f}{V_{k_2}}$.  Since $\trunc{f}{V_{k_1}}$ and $\trunc{f}{V_{k_2}}$ each have $N$ right contour points, Lemma~\ref{lem:refine trunc contours} also tells us that, if $\alpha_1^1$ is the first right contour point of $\trunc{f}{V_{k_1}}$ and $\alpha_1^2$ is the first right contour point of $\trunc{f}{V_{k_2}}$, then $\left| f(\alpha_1^1) \right| \leq \left| f(\alpha_1^2) \right|$.  That is, $|t_{k_1}(\frac{1}{N})| \leq |t_{k_2}(\frac{1}{N})|$.  Thus, the sequence $|t_k(\frac{1}{N})|$, $k \geq k_0$, is non-decreasing.

Let $k \in \mathbb{N}$ with $k_0 \leq k$.  Let $f \in \mathcal{F}_k$, and let $\alpha_1$ be the first right contour point of $\trunc{f}{V_k}$.  Since $t_k$ is the radial contour factor of $\trunc{f}{V_k}$, we have $t_k(\frac{1}{N}) = \trunc{f}{V_k}(\alpha_1) = f(\alpha_1)$, where the latter equality follows from Lemma~\ref{lem:deps trunc}.  By Lemma~\ref{lem:mesh preim contour},
\[ \mesh \left( t_k^{-1}(V_k) \cap [0,1] \right) \leq \frac{\mesh(V_k)}{N \cdot \left| f(\alpha_1) \right|} = \frac{\mesh(V_k)}{N \cdot \left| t_k(\frac{1}{N}) \right|} .\]
Therefore, since $\mesh(V_k) \to 0$ as $k \to \infty$, and $|t_k(\frac{1}{N})|$ is non-decreasing for $k \geq k_0$, we have that for sufficiently large $k$, $\mesh \left( t_k^{-1}(V_k) \cap [0,1] \right) < \delta$.
\end{proof}

\subsection{Proof of Theorem~\ref{thm:main}}

We now proceed with the proof of Theorem~\ref{thm:main}.  In this section we will make use of the concept of radial departures, and Propositions~\ref{prop:embed 0 accessible}, \ref{prop:comp dep}, and Lemma~\ref{lem:f tf same dep}, recalled in Section \ref{sec:rad deps}.

Let $X = \varprojlim \left \langle [-1,1],f_i \right \rangle$ be an arbitrary arc-like continuum (containing more than one point).  As mentioned in the preliminaries section above, we may assume that all of the bonding maps $f_i$ are piecewise-linear and nowhere constant.  Let $x = \langle x_i \rangle_{i=1}^\infty$ be an arbitrary point in $X$.  It is well-known that if $x_i = \pm 1$ for infinitely many $i$, then $x$ is an endpoint of $X$ and one can easily embed $X$ in $\mathbb{R}^2$ so as to make $x$ an accessible point (see e.g.\ \cite[p.295]{problems2018}).  Hence, we may as well assume that $x_i \neq \pm 1$ for all but finitely many $i$, and, in fact, by dropping finitely many coordinates we may assume that $x_i \neq \pm 1$ for all $i$.  Then, for each $i$, let $h_i \colon [-1,1] \to [-1,1]$ be a piecewise-linear homeomorphism such that $h_i(x_i) = 0$, and define $f_i' = h_i \circ f_i \circ h_{i+1}^{-1}$.  Then $\varprojlim \left \langle [-1,1],f_i' \right \rangle \approx X$, and, in fact, an explicit homeomorphism $h \colon X \to \varprojlim \left \langle [-1,1],f_i' \right \rangle$ is given by $h \left( \langle y_i \rangle_{i=1}^\infty \right) = \langle h_i(y_i) \rangle_{i=1}^\infty$, and clearly $h(x) = \langle 0,0,\ldots \rangle$.  Thus, to prove Theorem~\ref{thm:main}, it suffices to prove that if $f_i \colon [-1,1] \to [-1,1]$, $i = 1,2,\ldots$, are nowhere constant piecewise-linear maps such that $f_i(0) = 0$ for each $i$, there exists an embedding of $X = \varprojlim \left \langle [-1,1],f_i \right \rangle$ into $\mathbb{R}^2$ for which the point $\langle 0,0,\ldots \rangle \in X$ is accessible.

For the remainder of the proof, let $f_i \colon [-1,1] \to [-1,1]$, $i \in \mathbb{N}$ be nowhere constant piecewise-linear maps such that $f_i(0) = 0$ for all $i \in \mathbb{N}$.  According to Lemma~\ref{lem:nearly constant half}, if for infinitely many $i \in \mathbb{N}$, we have that for each $\varepsilon > 0$, there exists $n > i$ such that $f_i^n {\restriction}_{[0,1]}$ is $\varepsilon$-close to $0$ (constant map), then $X$ is homeomorphic to an inverse limit $Y = \varprojlim \left \langle [-1,1],g_k \right \rangle$, where $g_k \colon [-1,1] \to [-1,1]$, $k \in \mathbb{N}$, are maps such that $g_k {\restriction}_{[0,1]}$ is $0$ (constant map) for each $k \in \mathbb{N}$, and, in fact, there exists a homeomorphism which maps the point $\langle 0,0,\ldots \rangle \in X$ to $\langle 0,0,\ldots \rangle \in Y$.  Clearly these maps $g_k$ have no radial departures, therefore by Proposition~\ref{prop:embed 0 accessible}, there exists an embedding of $Y$ in $\mathbb{R}^2$ with the point $\langle 0,0,\ldots \rangle$ accessible.  Similar remarks apply for the restrictions $f_i^n {\restriction}_{[-1,0]}$.

In light of the previous paragraph, by dropping finitely many coordinates if necessary, we may suppose that for each $i \in \mathbb{N}$, there exists $\eta_i > 0$ such that the diameters of the sets $f_i^n([0,1])$ and $f_i^n([-1,0])$ are greater than or equal to $\eta_i$ for all $n > i$.  This means that for each $i \in \mathbb{N}$, if $V \subset [-1,1]$ is such that $0 \in V$ and $V$ has small enough mesh ($\mesh(V) \leq \eta_i$), then $\trunc{f_i^n}{V}$ has well-defined radial contour factor for all $n > i$.

Fix a sequence $\delta_1,\delta_3,\delta_5,\ldots$ of positive numbers such that $\delta_k \to 0$ as $k \to \infty$ (e.g.\ let $\delta_k = \frac{1}{k}$).  Our plan is to construct all of the elements of the diagram in Figure~\ref{fig:diagram} below, in such a way that the inverse limit of each row in the diagram will be homeomorphic to $X = \varprojlim \left \langle [-1,1],f_i \right \rangle$.  Specifically, the maps $F_k$ from the first row will be compositions of the maps $f_i$, so that $\varprojlim \left \langle [-1,1],F_k \right \rangle \approx X$ by standard properties of inverse limits (composing bonding maps).  The inverse limit of the second row in the diagram will be homeomorphic to the first by the criterion from Mioduszewski's Theorem, Theorem~\ref{thm:mioduszewski} above, for some sequence $\varepsilon_k$, $k \geq 1$ odd, of positive numbers to be constructed.  The inverse limits of the second and third rows are homeomorphic because the diagram commutes, i.e.\ $\trunc{F_k}{V_k} \circ \trunc{F_{k+1}}{V_{k+1}} = t_k \circ s_{k,k+1}$ for each odd $k \geq 1$.  Finally, the inverse limit of the fourth row in the diagram will be homeomorphic to the third by the criterion from Mioduszewski's Theorem, Theorem~\ref{thm:mioduszewski} above, for the sequence $\delta_k$, $k \geq 1$ odd.  The construction will be made so that in the final row of the diagram, the maps $s_{k,k+1} \circ \trunc{t_{k+2}}{V_{k+2}'}$ have no negative radial departures, so that the inverse limit can be embedded in $\mathbb{R}^2$ with the point $\langle 0,0,\ldots \rangle$ accessible, by Proposition~\ref{prop:embed 0 accessible}.

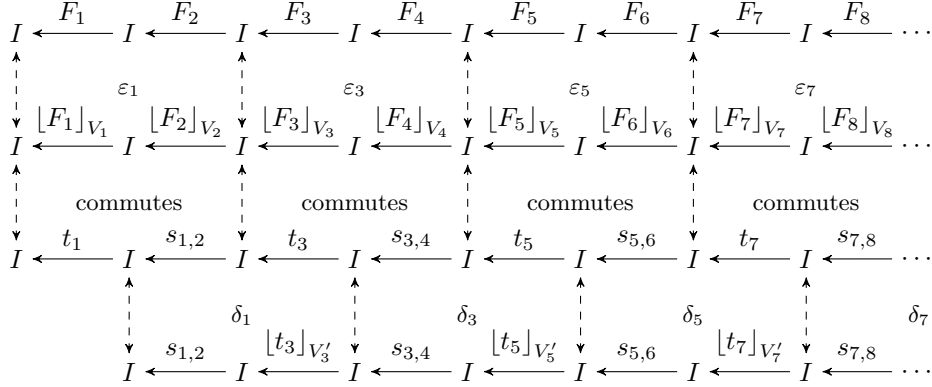
\begin{figure}
\begin{center}
\begin{tikzpicture}[>=stealth', scale=1.5]
\foreach \x in {1,...,8}
  \node (row1-\x) at (\x,4) {$I$};
\node (row1-9) at (9,4) {$\cdots$};

\foreach \x [remember=\x as \lastx (initially 1)] in {2,...,9}
  \draw[->] (row1-\x) to node[above]{$F_{\lastx}$} (row1-\lastx);

\foreach \x in {1,...,8}
  \node (row2-\x) at (\x,3) {$I$};
\node (row2-9) at (9,3) {$\cdots$};

\foreach \x [remember=\x as \lastx (initially 1)] in {2,...,9}
  \draw[->] (row2-\x) to node[above]{$\trunc{F_{\lastx}}{V_{\lastx}}$} (row2-\lastx);

\foreach \x in {1,3,...,7} {
  \draw[<->, dashed] (row1-\x)--(row2-\x);
  \node at (\x+1,3.5) {\small $\varepsilon_{\x}$};
}

\foreach \x in {1,...,8}
  \node (row3-\x) at (\x,2) {$I$};
\node (row3-9) at (9,2) {$\cdots$};

\foreach \x/\y in {1/2,3/4,5/6,7/8}
  \draw[->] (row3-\y) to node[above]{$t_{\x}$} (row3-\x);
\foreach \x/\y/\z in {1/2/3,3/4/5,5/6/7,7/8/9}
  \draw[->] (row3-\z) to node[above]{$s_{\x,\y}$} (row3-\y);

\foreach \x in {1,3,...,7} {
  \draw[<->, dashed] (row2-\x)--(row3-\x);
  \node at (\x+1,2.5) {\small commutes};
}

\foreach \x in {2,...,8}
  \node (row4-\x) at (\x,1) {$I$};
\node (row4-9) at (9,1) {$\cdots$};

\foreach \x/\y in {3/4,5/6,7/8}
  \draw[->] (row4-\y) to node[above]{$\trunc{t_{\x}}{V_{\x}'}$} (row4-\x);
\foreach \x/\y/\z in {1/2/3,3/4/5,5/6/7,7/8/9}
  \draw[->] (row4-\z) to node[above]{$s_{\x,\y}$} (row4-\y);

\foreach \x/\y in {1/2,3/4,5/6,7/8} {
  \draw[<->, dashed] (row3-\y)--(row4-\y);
  \node at (\y+1,1.5) {\small $\delta_{\x}$};
}
\end{tikzpicture}
\end{center}

\caption{Almost commutative diagram from the proof of Theorem~\ref{thm:main}.  Here, $I$ denotes the interval $[-1,1]$, and dashed arrows are identities.  The $\varepsilon_k$'s and $\delta_k$'s are placed in accordance with the convention described in Figure~\ref{fig:mioduszewski diagram} above.}
\label{fig:diagram}
\end{figure}

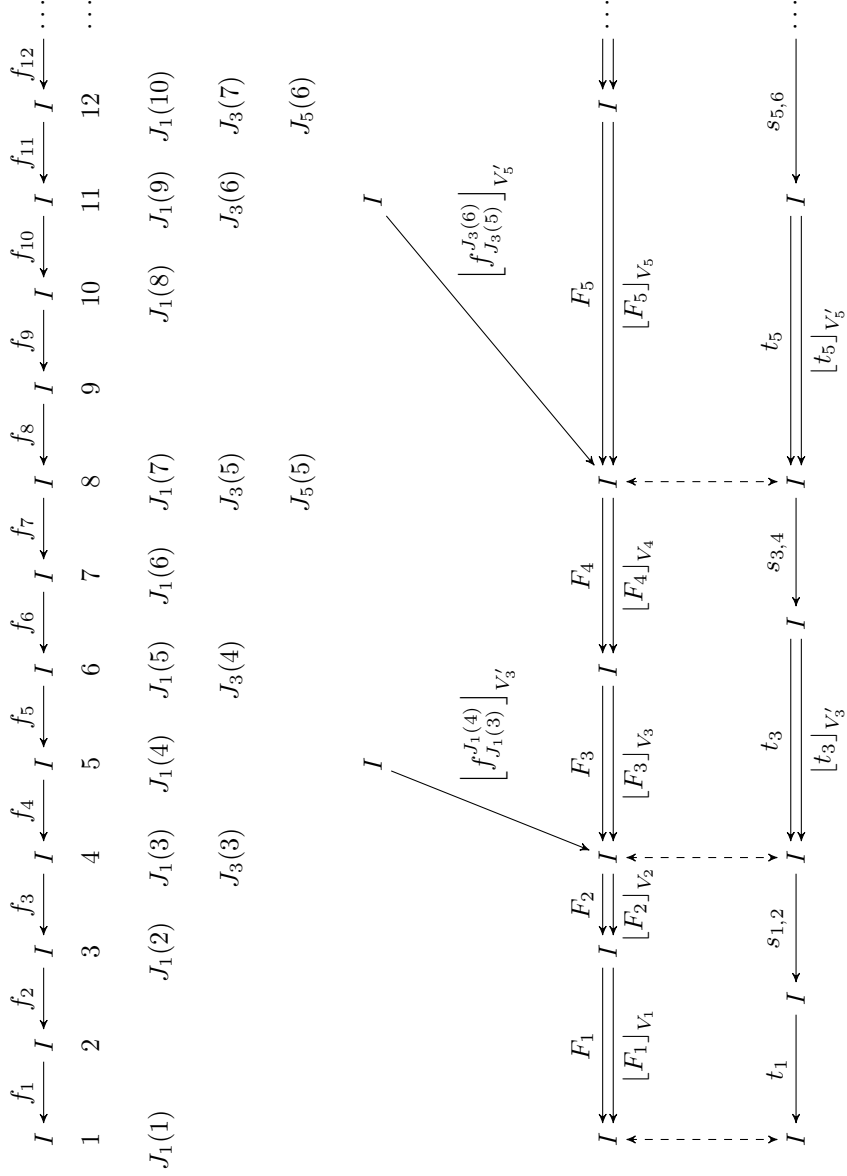
\begin{sidewaysfigure}
\begin{center}
\raisebox{-9.5in}{ 
\begin{tikzpicture}[>=stealth', scale=1.25]
\foreach \x in {1,...,12}
  \node (frow-\x) at (\x,0) {$I$};
\node (frow-13) at (13,0) {$\cdots$};

\foreach \x [remember=\x as \lastx (initially 1)] in {2,...,13}
  \draw[->] (frow-\x) to node[above]{$f_{\lastx}$} (frow-\lastx);

\foreach \x in {1,...,12}
  \node at (\x,-0.5) {$\x$};
\node at (13,-0.5) {$\cdots$};

\foreach \x [count=\xi] in {1,3,4,5,6,7,8,10,11,12}
  \node at (\x,-1.25) {$J_1(\xi)$};
\foreach \x [count=\xi] in {0,0,4,6,8,11,12} {
  \ifnum\xi>2
    \node at (\x,-2) {$J_3(\xi)$};
  \fi
}
\foreach \x [count=\xi] in {0,0,0,0,8,12} {
  \ifnum\xi>4
    \node at (\x,-2.75) {$J_5(\xi)$};
  \fi
}

\foreach \x [count=\xi] in {1,3,4,6,8,12}
  \node (crow-\xi) at (\x,-6) {$I$};
\node (crow-7) at (13,-6) {$\cdots$};

\foreach \x [remember=\x as \lastx (initially 1)] in {2,...,6} {
  \draw[->, transform canvas={yshift=2pt}] (crow-\x) to node[above]{$F_{\lastx}$} (crow-\lastx);
  \draw[->, transform canvas={yshift=-2pt}] (crow-\x) to node[below]{$\trunc{F_{\lastx}}{V_{\lastx}}$} (crow-\lastx);
}
\draw[->, transform canvas={yshift=2pt}] (crow-7) to (crow-6);
\draw[->, transform canvas={yshift=-2pt}] (crow-7) to (crow-6);

\node (orow-4) at (5,-3.5) {$I$};
\node (orow-6) at (11,-3.5) {$I$};
\draw[->] (orow-4) to node[pos=0.3, below right]{$\trunc{f_{J_1(3)}^{J_1(4)}}{V_3'}$} (crow-3);
\draw[->] (orow-6) to node[pos=0.3, below right]{$\trunc{f_{J_3(5)}^{J_3(6)}}{V_5'}$} (crow-5);

\foreach \x [count=\xi] in {1,2.5,4,6.5,8,11}
  \node (arow-\xi) at (\x,-8) {$I$};
  \node (arow-7) at (13,-8) {$\cdots$};

\draw[->] (arow-2) to node[above]{$t_1$} (arow-1);
\draw[->] (arow-3) to node[above]{$s_{1,2}$} (arow-2);
\draw[->, transform canvas={yshift=2pt}] (arow-4) to node[above]{$t_3$} (arow-3);
\draw[->, transform canvas={yshift=-2pt}] (arow-4) to node[below]{$\trunc{t_3}{V_3'}$} (arow-3);
\draw[->] (arow-5) to node[above]{$s_{3,4}$} (arow-4);
\draw[->, transform canvas={yshift=2pt}] (arow-6) to node[above]{$t_5$} (arow-5);
\draw[->, transform canvas={yshift=-2pt}] (arow-6) to node[below]{$\trunc{t_5}{V_5'}$} (arow-5);
\draw[->] (arow-7) to node[above]{$s_{5,6}$} (arow-6);

\foreach \x in {1,3,5}
  \draw[<->, dashed] (crow-\x) to (arow-\x);
\end{tikzpicture}
}
\end{center}

\caption{Alternative diagram depicting the elements to be constructed in the proof of Theorem~\ref{thm:main}, showing an example of what the sets $J_1,J_3,J_5$ might look like.}
\label{fig:diagram 2}
\end{sidewaysfigure}

Figure~\ref{fig:diagram 2} gives an alternative illustration of what will be constructed below, including some additional elements that are part of the construction but not shown in Figure~\ref{fig:diagram}.

We will recursively construct the following:
\begin{itemize}
\item Positive real numbers $\varepsilon_1,\varepsilon_3,\varepsilon_5,\ldots$ with $\varepsilon_k \to 0$ as $k \to \infty$;
\item Infinite subsets $J_1 \supset J_3 \supset J_5 \supset \cdots$ of $\mathbb{N}$;
\item Finite subsets $V_1,V_3,V_5,\ldots$ of $[-1,1]$ with $0 \in V_k$ for each odd $k \geq 1$; and
\item Maps $s_{k,k+1} \colon [-1,1] \to [-1,1]$, for each odd $k \geq 1$
\end{itemize}
satisfying the properties \ref{rec1} through \ref{rec8} below.  To state these properties, we use the following definitions/notation, for each odd $k \geq 1$:
\[ J_k = \{J_k(m): m \geq k\} \textrm{ where } J_k(k) < J_k(k+1) < J_k(k+2) < \cdots \]
\[ F_k = f_{J_k(k)}^{J_k(k+1)}, \; F_{k+1} = f_{J_k(k+1)}^{J_k(k+2)} \]
\[ V_{k+1} = F_k^{-1}(V_k), \; V_{k+2}' = F_{k+1}^{-1}(V_{k+1}) \]
\[ t_k = \textrm{the radial contour factor of } \trunc{F_k}{V_k} \]
\[ t_{k+2}' = \textrm{the radial contour factor of } \trunc{f_{J_k(k+2)}^{J_k(k+3)}}{V_{k+2}'} \]
The reader may wish to skim over the list of properties \ref{rec1} through \ref{rec8} below and refer to them in detail when they are referenced in the construction below.  The properties we will require are as follows.  For each odd $k \geq 1$:
\begin{enumerate}[label=(\alph{*})]
\item \label{rec1} For each odd $1 \leq j < k$, for any interval $A \subset [-1,1]$, if $\diam(A) < \varepsilon_k$ then $\diam(s_{j,j+1}(A)) < \delta_k$;
\item \label{rec2} For each odd $1 \leq j \leq k$, for any interval $A \subset [-1,1]$, if $\intr(A) \cap V_k = \emptyset$ then $\diam \left( F_j^k(A) \right) < \varepsilon_k$;
\item \label{rec3} $V_k \supseteq V_k'$ for each odd $k > 1$, and $\trunc{f_{J_k(k)}^n}{V_k}$ has well-defined radial contour factor for all $n > J_k(k)$, for each odd $k \geq 1$;
\item \label{rec4} $J_1(1) = 1$, and $J_k(k) = J_{k-2}(k)$ for each odd $k > 1$;
\item \label{rec5} $\mesh \left( t_k^{-1}(V_k) \right) < \delta_k$;
\item \label{rec6} For each $n_1,n_2 \in J_k$ with $n_1,n_2 > J_k(k+2)$, the radial contour factors of $\trunc{f_{J_k(k+2)}^{n_1}}{V_{k+2}'}$ and $\trunc{f_{J_k(k+2)}^{n_2}}{V_{k+2}'}$ (which are well-defined by \ref{rec3} and Lemma~\ref{lem:well-def preimage}) are the same (equal to $t_{k+2}'$);
\item \label{rec7} $t_k \circ s_{k,k+1} = \trunc{F_k}{V_k} \circ \trunc{F_{k+1}}{V_{k+1}}$; and
\item \label{rec8} $s_{k,k+1} \circ t_{k+2}'$ has no negative radial departures.
\end{enumerate}

To carry out this recursive construction, let $k \geq 1$ be odd and suppose that $V_k'$ and $J_{k-2}$ have been defined (in the case $k=1$, we may take $V_1' \supseteq \{-1,0,1\}$ with mesh small enough so that $\trunc{f_1^n}{V_1'}$ has well-defined radial contour factor for all $n > 1$, and $J_{-1} = \mathbb{N}$, with $J_{-1}(1) = 1$), and suppose also that the maps $s_{j,j+1}$, for odd $1 \leq j < k$, have been defined.  Refer to Figure~\ref{fig:construction step k} for a summary of what maps we will construct in this stage $k$.  Choose $\varepsilon_k > 0$ small enough so that \ref{rec1} holds (in the case $k=1$, we may choose $\varepsilon_1 > 0$ arbitrarily).  Let $J_k(k) = J_{k-2}(k)$.  Then apply Lemma~\ref{lem:V J} to the family
\[ \mathcal{F} = \left\{ f_{J_k(k)}^n: n \in J_{k-2} \textrm{ and } n > J_k(k) \right\} \]
and $V = V_k'$ to obtain a finite set $V_k$ and an infinite subset $J_k' \subset J_{k-2}$ such that:
\begin{itemize}
\item \ref{rec3} is satisfied;
\item $n > J_k(k)$ for all $n \in J_k'$;
\item all of the maps $\trunc{f_{J_k(k)}^n}{V_k}$, $n \in J_k'$, have the same radial contour factor $t_k$; and
\item \ref{rec5} is satisfied.
\end{itemize}

We may further assume that $\mesh(V_k)$ is small enough so that \ref{rec2} holds.  Let $J_k(k+1) = \min J_k'$, and let $V_{k+1} = \left( f_{J_k(k)}^{J_k(k+1)} \right)^{-1}(V_k)$.  By Lemma~\ref{lem:well-def preimage}, $\trunc{f_{J_k(k+1)}^n}{V_{k+1}}$ has well-defined radial contour factor for all $n > J_k(k+1)$.  Next, apply Lemma~\ref{lem:same trunc contour} to the family
\[ \mathcal{F} = \left\{ f_{J_k(k+1)}^n: n \in J_k' \textrm{ and } n > J_k(k+1) \right\} \]
to obtain an infinite set $J_k'' \subseteq J_k' \smallsetminus \{J_k(k+1)\}$ such that all of the maps $\trunc{f_{J_k(k+1)}^n}{V_{k+1}}$, $n \in J_k''$, have the same radial contour factor.  Let $J_k(k+2) = \min J_k''$, and let $V_{k+2}' = \left( f_{J_k(k+1)}^{J_k(k+2)} \right)^{-1}(V_{k+1})$.  By Lemma~\ref{lem:well-def preimage}, $\trunc{f_{J_k(k+2)}^n}{V_{k+2}'}$ has well-defined radial contour factor for all $n > J_k(k+2)$.  Now apply Lemma~\ref{lem:same trunc contour} again, this time to the family
\[ \mathcal{F} = \left\{ f_{J_k(k+2)}^n: n \in J_k'' \textrm{ and } n > J_k(k+2) \right\} ,\]
to obtain an infinite set $J_k''' \subseteq J_k'' \smallsetminus \{J_k(k+2)\}$ such that all of the maps $\trunc{f_{J_k(k+2)}^n}{V_{k+2}'}$, $n \in J_k'''$, have the same radial contour factor.  Let $J_k = \{J_k(k), J_k(k+1), J_k(k+2)\} \cup J_k'''$.  By construction, and the defining property of $J_k'''$, we have that \ref{rec4} and \ref{rec6} are satisfied.

Finally, we will construct the map $s_{k,k+1}$ using Lemma~\ref{lem:bridged s} as follows.  Note that by the third bullet point above in the properties of $J_k'$, the radial contour factors of
\[ \trunc{F_k}{V_k} = \trunc{f_{J_k(k)}^{J_k(k+1)}}{V_k} \]
and of
\begin{align*}
\trunc{F_k}{V_k} \circ \trunc{F_{k+1}}{V_{k+1}} &= \trunc{F_k \circ F_{k+1}}{V_k} \quad \textrm{(see Lemma~\ref{lem:compare trunc}\ref{comp trunc})} \\
&= \trunc{f_{J_k(k)}^{J_k(k+2)}}{V_k}
\end{align*}
are the same (equal to $t_k$).  Moreover, the radial contour factors of
\[ \trunc{F_{k+1}}{V_{k+1}} = \trunc{f_{J_k(k+1)}^{J_k(k+2)}}{V_{k+1}} \]
and of
\begin{align*}
\trunc{F_{k+1}}{V_{k+1}} \circ \trunc{f_{J_k(k+2)}^{J_k(k+3)}}{V_{k+2}'} &= \trunc{F_{k+1} \circ f_{J_k(k+2)}^{J_k(k+3)}}{V_{k+1}} \quad \textrm{(see Lemma~\ref{lem:compare trunc}\ref{comp trunc})} \\
&= \trunc{f_{J_k(k+1)}^{J_k(k+3)}}{V_{k+1}}
\end{align*}
are the same by the defining property of $J_k''$.  Therefore, the existence of $s_{k,k+1}$ satisfying \ref{rec7} and \ref{rec8} is given by Lemma~\ref{lem:bridged s} using, in the notation of that Lemma, $\trunc{F_k}{V_k}$ for $f_1$, $\trunc{F_{k+1}}{V_{k+1}}$ for $f_2$, and $\trunc{f_{J_k(k+2)}^{J_k(k+3)}}{V_{k+2}'}$ for $f_3$.  This completes the recursive construction.

\begin{figure}
\begin{center}
\begin{tikzpicture}[>=stealth', scale=1.5]
\node (row1-1) at (0,3) {$I$};
\node (row1-2) at (1.5,3) {$I$};
\node (row1-3) at (3,3) {$I$};

\draw[->] (row1-2) to node[above]{$F_k$} (row1-1);
\draw[->] (row1-3) to node[above]{$F_{k+1}$} (row1-2);

\node (row2-1) at (0,2) {$I$};
\node (row2-2) at (1.5,2) {$I$};
\node (row2-3) at (3,2) {$I$};
\node (row2-4) at (5.5,2) {\textcolor{red}{$I$}};

\draw[->] (row2-2) to node[above]{$\trunc{F_k}{V_k}$} (row2-1);
\draw[->] (row2-3) to node[above]{$\trunc{F_{k+1}}{V_{k+1}}$} (row2-2);
\draw[->, red] (row2-4) to node[above]{$\trunc{f_{J_k(k+2)}^{J_k(k+3)}}{V_{k+2}'}$} (row2-3);

\draw[<->, dashed] (row1-1)--(row2-1);
\draw[<->, dashed] (row1-3)--(row2-3);

\node (row3-1) at (0,1) {$I$};
\node (row3-2) at (1.5,1) {$I$};
\node (row3-3) at (3,1) {$I$};
\node (row3-4) at (4.5,1) {\textcolor{red}{$I$}};

\draw[->] (row3-2) to node[above]{$t_k$} (row3-1);
\draw[->] (row3-3) to node[above]{$s_{k,k+1}$} (row3-2);
\draw[->, red] (row3-4) to node[above]{$t_{k+2}'$} (row3-3);

\draw[<->, dashed] (row2-1)--(row3-1);
\draw[<->, dashed] (row2-3)--(row3-3);

\node at (1.5,1.5) {\small commutes};

\node (row4-1) at (0,0) {$I$};
\node (row4-2) at (1.5,0) {$I$};
\node (row4-3) at (3,0) {$I$};

\draw[->] (row4-2) to node[above]{$\trunc{t_k}{V_k'}$} (row4-1);
\draw[->] (row4-3) to node[above]{$s_{k,k+1}$} (row4-2);

\draw[<->, dashed] (row3-2)--(row4-2);
\end{tikzpicture}
\end{center}

\caption{Maps constructed in the $k$-th stage of the recursion in the proof of Theorem~\ref{thm:main}.  Here, $I$ denotes the interval $[-1,1]$, and dashed arrows are identities.  The maps in red, i.e.\ $\trunc{f_{J_k(k+2)}^{J_k(k+3)}}{V_{k+2}'}$ and $t_{k+2}'$, are not part of the diagram in Figure~\ref{fig:diagram}, even though they are part of the recursive construction.}
\label{fig:construction step k}
\end{figure}
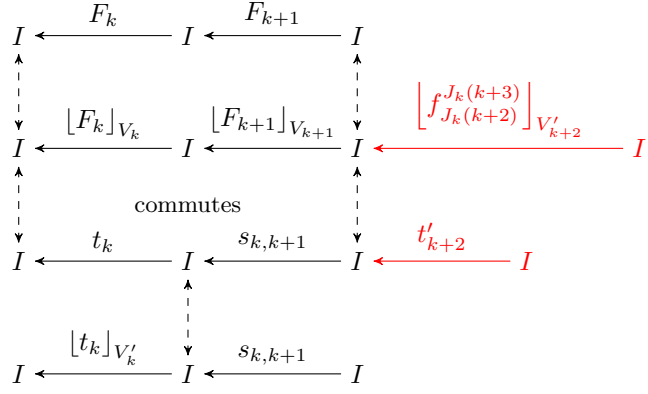

\begin{claim}
\label{claim:rows 12}
Let $G_k = \trunc{F_k}{V_k}$ for each $k \geq 1$.  For every odd $j \leq k \leq \ell$,
\begin{enumerate}
\item $F_j^\ell$ and $F_j^k \circ G_k^\ell$ are $\varepsilon_k$-close; and
\item $G_j^\ell$ and $G_j^k \circ F_k^\ell$ are $\varepsilon_k$-close.
\end{enumerate}
\end{claim}

\begin{proof}[Proof of Claim~\ref{claim:rows 12}]
\renewcommand{\qedsymbol}{\textsquare (Claim~\ref{claim:rows 12})}
First we prove the following:
\begin{enumerate}[label=($\ast$)]
\item \label{F G V-close} For every odd $k \leq \ell$, $F_k^\ell$ and $G_k^\ell$ are $V_k$-close.
\end{enumerate}
To prove \ref{F G V-close}, it suffices to prove that for every odd $p$ and all $x,y \in [-1,1]$, if $x \overset{V_{p+2}}{=\joinrel=} y$ then $F_p \circ F_{p+1}(x) \overset{V_p}{=} G_p \circ G_{p+1}(y)$ (we deduce \ref{F G V-close} by induction, starting with $x \in [-1,1]$, we have $x \overset{V_\ell}{=\joinrel=} x$, hence $F^\ell_{\ell-2}(x) \overset{V_{\ell-2}}{=\joinrel=} G^\ell_{\ell-2}(x)$, hence $F^\ell_{\ell-4}(x) \overset{V_{\ell-4}}{=\joinrel=} G^\ell_{\ell-4}(x)$, etc.).  To this end, let $p$ be odd and suppose $x,y \in [-1,1]$ and $x \overset{V_{p+2}}{=\joinrel=} y$.  Recall that $G_p \circ G_{p+1} = \trunc{F_p \circ F_{p+1}}{V_p}$.  Since $V_{p+2} \supseteq V_{p+2}' = (F_p \circ F_{p+1})^{-1}(V_p)$, we have $x \overset{V_{p+2}'}{=\joinrel=} y$, and so by Lemma~\ref{lem:V-close props}\ref{f(x) truncf(y) V-close}, $F_p \circ F_{p+1}(x) \overset{V_p}{=} G_p \circ G_{p+1}(y)$, as desired.

\medskip
Now, let $j \leq k \leq \ell$ be odd and let $x \in [-1,1]$.  Let $A$ denote the interval with endpoints $F_k^\ell(x)$ and $G_k^\ell(x)$ (or $A = \{F_k^\ell(x)\}$ in the case that $F_k^\ell(x) = G_k^\ell(x)$).  By \ref{F G V-close}, $F_k^\ell(x) \overset{V_k}{=} G_k^\ell(x)$, which means that $\intr(A) \cap V_k = \emptyset$.  Then by property \ref{rec2} above, $\diam(F_j^k(A)) < \varepsilon_k$, so in particular $F_j^\ell(x) = F_j^k \circ F_k^\ell(x)$ and $F_j^k \circ G_k^\ell(x)$ are $\varepsilon_k$-close.  This proves (1).

Moreover, each map in the composition $G_j^k$ is a truncation of the corresponding map in $F_j^k$, so according to Lemma~\ref{lem:interval image trunc} (applied repeatedly), either $G_j^k(A)$ is degenerate or $G_j^k(A) \subseteq F_j^k(A)$.  In either case, $\diam(G_j^k(A)) \leq \diam(F_j^k(A)) < \varepsilon_k$, so in particular $G_j^\ell(x) = G_j^k \circ G_k^\ell(x)$ and $G_j^k \circ F_k^\ell(x)$ are $\varepsilon_k$-close.  This proves (2).
\end{proof}

It follows from Claim~\ref{claim:rows 12} and Theorem~\ref{thm:mioduszewski} that $\varprojlim \left \langle [-1,1],F_k \right \rangle$ and $\varprojlim \left \langle [-1,1],\trunc{F_k}{V_k} \right \rangle$ are homeomorphic, and the homeomorphism between them, provided by Theorem~\ref{thm:mioduszewski}, maps $\langle 0,0,\ldots \rangle$ to $\langle 0,0,\ldots \rangle$.  By property \ref{rec7} above, this latter inverse limit is homeomorphic to
\[ \varprojlim \left \langle [-1,1], s_{1,2}, [-1,1], t_3, [-1,1], s_{3,4}, [-1,1], t_5, \ldots \right \rangle ,\]
and, again, the (natural) homeomorphism between them maps $\langle 0,0,\ldots \rangle$ to $\langle 0,0,\ldots \rangle$.

\begin{claim}
\label{claim:rows 34}
Let $\varphi_k = s_{k,k+1} \circ t_{k+2}$ and $\gamma_k = s_{k,k+1} \circ \trunc{t_{k+2}}{V_{k+2}'}$ for each odd $k \geq 1$.  Also, for odd $k < \ell$, denote $\varphi_k^\ell = \varphi_k \circ \varphi_{k+2} \circ \cdots \circ \varphi_{\ell-2}$ and $\gamma_k^\ell = \gamma_k \circ \gamma_{k+2} \circ \cdots \circ \gamma_{\ell-2}$.  For every odd $j \leq k \leq \ell$,
\begin{enumerate}
\item $\varphi_j^\ell$ and $\varphi_j^k \circ \gamma_k^\ell$ are $\delta_k$-close; and
\item $\gamma_j^\ell$ and $\gamma_j^k \circ \varphi_k^\ell$ are $\delta_k$-close.
\end{enumerate}
\end{claim}

\begin{proof}[Proof of Claim~\ref{claim:rows 34}]
\renewcommand{\qedsymbol}{\textsquare (Claim~\ref{claim:rows 34})}
We point out that for each odd $k < \ell$, $\varphi_k^\ell = s_{k,k+1} \circ F_{k+2}^\ell \circ t_\ell$.

First we prove the following:
\begin{enumerate}[label=($\ast\ast$)]
\item \label{phi gamma tinvV-close} For every odd $k \leq \ell$, $\varphi_k^\ell$ and $\gamma_k^\ell$ are $t_k^{-1}(V_k)$-close.
\end{enumerate}
To prove \ref{phi gamma tinvV-close}, it suffices (as with \ref{F G V-close} in the proof of Claim~\ref{claim:rows 12}) to prove that for every odd $p$ and all $x,y \in [-1,1]$, if $x \overset{t_{p+2}^{-1}(V_{p+2})}{=\joinrel=} y$ then $\varphi_p(x) \overset{t_p^{-1}(V_p)}{=\joinrel=} \gamma_p(y)$.  To this end, let $p$ be odd and suppose $x,y \in [-1,1]$ and $x \overset{t_{p+2}^{-1}(V_{p+2})}{=\joinrel=} y$.  Since $V_{p+2}' \subseteq V_{p+2}$, we have $x \overset{t_{p+2}^{-1}(V_{p+2}')}{=\joinrel=} y$, so by Lemma~\ref{lem:V-close props}\ref{f(x) truncf(y) V-close}, $t_{p+2}(x) \overset{V_{p+2}'}{=\joinrel=} \trunc{t_{p+2}}{V_{p+2}'}(y)$.  Recall that $V_{p+2}' = (F_p \circ F_{p+1})^{-1}(V_p)$, and that $\trunc{F_p \circ F_{p+1}}{V_p} = t_p \circ s_{p,p+1}$, so by Lemma~\ref{lem:f s t V-close} (applied using, in the notation of that Lemma, $V_p$ for $V$, $F_p \circ F_{p+1}$ for $f$, $t_p$ for $t$, $s_{p,p+1}$ for $s$, $t_{p+2}(x)$ for $x$, and $\trunc{t_{p+2}}{V_{p+2}'}(y)$ for $y$), $s_{p,p+1} \circ t_{p+2}(x) \overset{t_p^{-1}(V_p)}{=\joinrel=} s_{p,p+1} \circ \trunc{t_{p+2}}{V_{p+2}'}(y)$, i.e.\ $\varphi_p(x) \overset{t_p^{-1}(V_p)}{=\joinrel=} \gamma_p(y)$, as desired.

\medskip
Now, let $j \leq k \leq \ell$ be odd and let $x \in [-1,1]$.  If $j = k$, then by \ref{phi gamma tinvV-close} and property \ref{rec5} above, $\varphi_j^\ell = \varphi_k^\ell$ and $\varphi_j^k \circ \gamma_k^\ell = \gamma_k^\ell$ are $\delta_k$-close.  Suppose, then, that $j < k$.  Let $A$ denote the interval with endpoints $\varphi_k^\ell(x)$ and $\gamma_k^\ell(x)$ (or $A = \{\varphi_k^\ell(x)\}$ in the case that $\varphi_k^\ell(x) = \gamma_k^\ell(x)$).  By \ref{phi gamma tinvV-close}, $\varphi_k^\ell(x) \overset{t_k^{-1}(V_k)}{=\joinrel=} \gamma_k^\ell(x)$, which means that $\intr(A) \cap t_k^{-1}(V_k) = \emptyset$.  It follows (e.g.\ from Lemma~\ref{lem:V-close props}\ref{f(x) f(y) V-close} applied to each pair of points in $A$) that $\intr(t_k(A)) \cap V_k = \emptyset$.  Then by property \ref{rec2} above, $\diam(F_{j+2}^k \circ t_k(A)) < \varepsilon_k$, so by property \ref{rec1} above, $\diam(s_{j,j+1} \circ F_{j+2}^k \circ t_k(A)) < \delta_k$, i.e.\ $\diam(\varphi_j^k(A)) < \delta_k$.  So in particular, the points
\[ \varphi_j^\ell(x) = \varphi_j^k \circ \varphi_k^\ell(x) = s_{j,j+1} \circ F_{j+2}^k \circ t_k \circ \varphi_k^\ell(x) \]
and
\[ \varphi_j^k \circ \gamma_k^\ell(x) = s_{j,j+1} \circ F_{j+2}^k \circ t_k \circ \gamma_k^\ell(x) \]
are $\delta_k$-close.  This proves (1).

Moreover, note that each map in the composition
\[ \gamma_j^k = s_{j,j+1} \circ \trunc{t_{j+2}}{V_{j+2}'} \circ \cdots \circ s_{k-2,k-1} \circ \trunc{t_k}{V_k'} \]
is either equal to the corresponding map in $\varphi_j^k$ or a truncation of the corresponding map in $\varphi_j^k$, so according to Lemma~\ref{lem:interval image trunc} (applied repeatedly), either $\gamma_j^k(A)$ is degenerate or $\gamma_j^k(A) \subseteq \varphi_j^k(A)$.  In either case, $\diam(\gamma_j^k(A)) \leq \diam(\varphi_j^k(A)) < \delta_k$, so in particular $\gamma_j^\ell(x) = \gamma_j^k \circ \gamma_k^\ell(x)$ and $\gamma_j^k \circ \varphi_k^\ell(x)$ are $\delta_k$-close.  This proves (2).
\end{proof}

It follows from Claim~\ref{claim:rows 34} and Theorem~\ref{thm:mioduszewski} that
\[ \varprojlim \left \langle [-1,1], s_{1,2}, [-1,1], t_3, [-1,1], s_{3,4}, [-1,1], t_5, \ldots \right \rangle \]
and
\[ \varprojlim \left \langle [1,1], s_{1,2}, [-1,1], \trunc{t_3}{V_3'}, [-1,1], s_{3,4}, [-1,1], \trunc{t_5}{V_5'}, \ldots \right \rangle \]
are homeomorphic, and the homeomorphism between them, provided by Theorem~\ref{thm:mioduszewski}, maps $\langle 0,0,\ldots \rangle$ to $\langle 0,0,\ldots \rangle$.

\begin{claim}
\label{claim:no neg deps}
For each odd $k \geq 1$, $s_{k,k+1} \circ \trunc{t_{k+2}}{V_{k+2}'}$ has no negative radial departures.
\end{claim}

\begin{proof}[Proof of Claim~\ref{claim:no neg deps}]
\renewcommand{\qedsymbol}{\textsquare (Claim~\ref{claim:no neg deps})}
Recall that according to property \ref{rec4} above, $J_{k+2}(k+2) = J_k(k+2)$.  This means $F_{k+2} = f_{J_{k+2}(k+2)}^{J_{k+2}(k+3)} = f_{J_k(k+2)}^{J_{k+2}(k+3)}$, and $J_{k+2}(k+3) \in J_{k+2} \subset J_k$ with $J_{k+2}(k+3) > J_{k+2}(k+2) = J_k(k+2)$.  So by property \ref{rec6}, $t_{k+2}'$ is the radial contour factor of $\trunc{F_{k+2}}{V_{k+2}'}$.  Recall also that by property \ref{rec8} above, $s_{k,k+1} \circ t_{k+2}'$ has no negative radial departures.

By definition, $t_{k+2}$ is the radial contour factor of $\trunc{F_{k+2}}{V_{k+2}}$.  Since $V_{k+2} \supseteq V_{k+2}'$ (property \ref{rec3} above), by Lemma~\ref{lem:compare trunc}\ref{fV tfWV} we have that $t_{k+2}'$ is the radial contour factor of $\trunc{t_{k+2}}{V_{k+2}'}$.  Therefore, by Lemma~\ref{lem:f tf same dep}, $t_{k+2}'$ and $\trunc{t_{k+2}}{V_{k+2}'}$ have the same radial departures, so by Lemma~\ref{lem:comp same dep}, $s_{k,k+1} \circ t_{k+2}'$ and $s_{k,k+1} \circ \trunc{t_{k+2}}{V_{k+2}'}$ have the same radial departures.  Thus, $s_{k,k+1} \circ \trunc{t_{k+2}}{V_{k+2}'}$ also has no negative radial departures.
\end{proof}

It follows from Claim~\ref{claim:no neg deps} and Proposition~\ref{prop:embed 0 accessible} that the inverse limit
\[ \varprojlim \left \langle [-1,1], s_{1,2}, [-1,1], \trunc{t_3}{V_3'}, [-1,1], s_{3,4}, [-1,1], \trunc{t_5}{V_5'}, \ldots \right \rangle \]
can be embedded in $\mathbb{R}^2$ with the point $\langle 0,0,\ldots \rangle$ accessible.  This completes the proof of Theorem~\ref{thm:main}.

\section{Discussion and questions}
\label{sec:discussion}

It is worth noting that the embedding of $\varprojlim \left \langle [-1,1],f_i \right \rangle$ produced in the proof of Proposition~\ref{prop:embed 0 accessible} in \cite{ammerlaan-anusic-hoehn2023} comes from the Anderson-Choquet Embedding Theorem \cite{anderson-choquet1959} applied to the inverse system $\langle [-1,1],f_i \rangle$.  As such, the embedding may be assumed to be \emph{thin} (see e.g.\ \cite{debski-tymchatyn1993}), meaning that for any $\varepsilon > 0$, there is a chain of connected open sets in the plane, of diameters less than $\varepsilon$, covering the image of the embedding.  Note that there exist embeddings of arc-like continua in $\mathbb{R}^2$ which are not thin, see e.g.\ \cite{bing1951} and \cite{debski-tymchatyn1993}.  More details on this sharper version Proposition~\ref{prop:embed 0 accessible}, as well as applications related to the original motivation of Nadler and Quinn to study accessible points of arc-like continua, are given in \cite{ammerlaan-hoehn2024}.

\medskip

Given a planar continuum $X$ and two embeddings $\Phi_1,\Phi_2 \colon X \to \mathbb{R}^2$ of $X$ into $\mathbb{R}^2$, we say $\Phi_1$ and $\Phi_2$ are \emph{equivalent embeddings} if there exists a homeomorphism $h \colon \mathbb{R}^2 \to \mathbb{R}^2$ such that $\Phi_2 = h \circ \Phi_1$.  Clearly, if $\Phi_1$ and $\Phi_2$ are equivalent embeddings, then for any point $x \in X$, $\Phi_1(x)$ is an accessible point of $\Phi_1(X)$ if and only if $\Phi_2(x)$ is an accessible point of $\Phi_2(X)$.

Recall that a continuum is \emph{indecomposable} if it is not the union of two of its proper subcontinua.  Mayer \cite[Question~4.2.2]{mayer1982} asked in 1982, is it true that for every indecomposable arc-like continuum $X$, there exists an uncountable family of pairwise inequivelant embeddings of $X$ in $\mathbb{R}^2$?  This question has also appeared in \cite[Problem~141]{lewis1983} and \cite[Question~1]{anusic-bruin-cinc2017}.  These authors establish some partial results for certain arc-like continua, namely Knaster continua in \cite{mayer1982}, the pseudo-arc in \cite{lewis1981}, and arc-like continua which are inverse limits of a single unimodal bonding map with positive topological entropy in \cite{anusic-bruin-cinc2017}.  From Theorem~\ref{thm:main} we now deduce a complete affirmative answer to this question.

\begin{cor}
\label{cor:ineq embeddings}
Let $X$ be an arc-like continuum which contains a non-degenerate indecomposable subcontinuum.  Then there are uncountably many pairwise inequivalent embeddings of $X$ in $\mathbb{R}^2$.
\end{cor}

\begin{proof}
Suppose $Y$ is a non-degenerate indecomposable subcontinuum of $X$.  For each $p \in Y$, let $\Phi_p \colon X \to \mathbb{R}^2$ be an embedding of $X$ in $\mathbb{R}^2$ such that $\Phi_p(p)$ is an accessible point of $\Phi_p(X)$.  Given $p \in Y$, let $A_p$ be the set of all points $y \in Y$ such that $\Phi_p(y)$ is an accessible point of $\Phi_p(Y)$.  Then $p \in A_p$, and $A_p$ is a meager subset of $Y$ by \cite{mazurkiewicz1929}, since $Y$ is indecomposable.  Given $p_1,p_2 \in Y$, if $\Phi_{p_1}$ and $\Phi_{p_2}$ are equivalent embeddings, then $A_{p_1} = A_{p_2}$.  By the Baire Category Theorem, there are uncountably many different sets among the sets $A_p$, $p \in Y$.  Therefore, there are uncountably many pairwise inequivalent embeddings of $X$ in $\mathbb{R}^2$.
\end{proof}

The analogous question about inequivalent embeddings of \emph{hereditarily decomposable} arc-like continua, i.e.\ those which do not contain a non-degenerate indecomposable subcontinuum, remains open.

\begin{question}[Question~6 of \cite{anusic-bruin-cinc2020}]
Let $X$ be a hereditarily decomposable arc-like continuum which is not an arc.  Does there exist an uncountable family of pairwise inequivelant embeddings of $X$ in $\mathbb{R}^2$?
\end{question}

Again, partial results have been established for certain arc-like continua, namely the $\sin(\frac{1}{x})$-continuum in \cite{mayer1982}, and the Knaster $V\Lambda$-continuum in \cite{ozbolt2020}.

\medskip

By a \emph{pointed space}, we mean a pair $\langle X,x \rangle$ where $X$ is a space and $x \in X$.  Given finitely many pointed spaces $\langle X_i,x_i \rangle$, $i = 1,\ldots,n$, the \emph{wedge sum} of these pointed spaces is the quotient of the disjoint union $X_1 \sqcup \cdots \sqcup X_n$ by identifying the set $\{x_i,\ldots,x_n\}$ to a single point (here it is understood that $x_i$ is taken to be the distinguished point from the $i$-th space in the disjoint union $X_1 \sqcup \cdots \sqcup X_n$).  Given a positive integer $m \geq 3$, a \emph{simple $m$-od} is a space which is the union of $m$ arcs, which all have one common endpoint, and which are otherwise pairwise disjoint.  In other words, a simple $m$-od is homeomorphic to the wedge sum of $m$ copies of the pointed space $\langle [0,1],0 \rangle$.

Suppose $X_i$, $i = 1,\ldots,n$, are finitely many arc-like continua, and $x_i$ is a point of $X_i$ for each $i$.  It is straightforward to see that the wedge sum of the pointed spaces $\langle X_i,x_i \rangle$ is a simple-$(2n)$-od-like continuum, i.e.\ is homeomorphic to an inverse limit of simple $(2n)$-ods.  We observe in the next Corollary that each continuum constructed this way is planar.

\begin{cor}
Let $X_i$, $i = 1,\ldots,n$, be finitely many arc-like continua, and let $x_i$ be an arbitrarily chosen point from $X_i$ for each $i = 1,\ldots,n$.  Then the wedge sum of the pointed spaces $\langle X_i,x_i \rangle$, $i = 1,\ldots,n$, is a planar continuum.
\end{cor}

\begin{proof}
Apply Theorem~\ref{thm:main} to each of the arc-like continua to obtain embeddings $\Phi_i \colon X_i \to \mathbb{R}^2$, $i = 1,\ldots,n$, of $X_i$ into $\mathbb{R}^2$ such that $\Phi_i(x_i)$ is an accessible point of $\Phi_i(X_i)$.  We may assume the images $\Phi_i(X_i)$, $i = 1,\ldots,n$, are pairwise disjoint.  From a point $z \in \mathbb{R}^2 \smallsetminus \left( \Phi_1(X_1) \cup \cdots \cup \Phi_n(X_n) \right)$, we extend an arc $L_i$ from $z$ to $\Phi_i(x_i)$ which is otherwise disjoint from $\Phi_1(X_1) \cup \cdots \cup \Phi_n(X_n)$, and such that the arcs $L_i$ are, aside from their common point $z$, otherwise pairwise disjoint.  The union of these arcs $L_1 \cup \cdots \cup L_n$ is a simple $n$-od.  We take the quotient of $\mathbb{R}^2$ by collapsing $L_1 \cup \cdots \cup L_n$ to a single point.  The result is again homeomorphic to $\mathbb{R}^2$ (e.g.\ by Moore's Theorem \cite{moore1925}), and the image of $\Phi_1(X_1) \cup \cdots \cup \Phi_n(X_n)$ under the quotient projection is the wedge sum of the pointed spaces $\langle X_i,x_i \rangle$, $i = 1,\ldots,n$.
\end{proof}

\medskip

After Theorem~\ref{thm:main}, a natural next question is whether the same result can be established for multiple points simultaneously.

\begin{question}
Let $X$ be an arc-like continuum, and let $S \subset X$ be a finite subset of $X$.  Does there exist an embedding of $X$ in $\mathbb{R}^2$ for which all of the points of $S$ are accessible?
\end{question}

Equivalently, if a continuum is formed by attaching finitely many arcs at one of their endpoints to an arc-like continuum $X$, where the arcs are disjoint from each other and otherwise disjoint from $X$, is the resulting tree-like continuum planar?

\bibliographystyle{amsplain}
\bibliography{Truncations}

\end{document}